\documentclass[12pt]{amsart}
\usepackage{mathrsfs}
\usepackage[cmtip,line,color,poly,all]{xy}
\usepackage{graphicx,amssymb,booktabs, verbatim, mathtools}
\usepackage[noadjust]{cite}
\usepackage{xspace}

\usepackage[pdfusetitle, pdfpagelabels, plainpages=false,
  bookmarks, bookmarksnumbered]{hyperref}

\usepackage{bookmark}
\usepackage{subfigure}
\newif\iffloattoend

\iffloattoend
\usepackage[figuresonly, nomarkers, nolists, heads]{endfloat}

\fi

\DeclarePairedDelimiter{\set}{\{}{\}}

\newcommand{\torconst}{t^{(2)}_D(\triv)}
\newcommand{\RR}{\mathbb R}%Reals
\newcommand{\ZZ}{\mathbb Z}%Integers
\newcommand{\QQ}{\mathbb Q}%Rationals
\newcommand{\CC}{\mathbb C}%Complex
\newcommand{\FF}{\mathbb F}%Complex
\newcommand{\Qbar}{\overline{\QQ}}
\newcommand{\OO}{\mathscr O}%Ring of integers

\newcommand{\sM}{\mathscr{M}}
\newcommand{\sL}{\mathscr{L}}

\newcommand{\cS}{\mathcal{S}}

\DeclareMathOperator{\vol}{vol}
\DeclareMathOperator{\triv}{triv}
\DeclareMathOperator{\tors}{\mathrm{tors}}
\newcommand{\fn}{\mathfrak{n}}

\newcommand{\fp}{\mathfrak{p}}

\newcommand{\HH}{\mathbb{H}}

\newcommand{\PP}{\mathbb{P}}
\newcommand{\sengun}{{\c{S}}eng{\"u}n\xspace}
\newcommand{\group}[1]{\mathbf{#1}}%Algebraic groups
\newcommand{\bG}{\group{G}}
\DeclareMathOperator{\reg}{reg}%regulator
\DeclareMathOperator{\SL}{SL}%Special linear
\DeclareMathOperator{\SO}{SO}%Special Orthogonal
\DeclareMathOperator{\Orth}{O}%Special linear
\DeclareMathOperator{\GL}{GL}%General linear
\DeclareMathOperator{\PGL}{PGL}%projective General linea
\DeclareMathOperator{\PSL}{PSL}%projective special linea
\DeclareMathOperator{\SU}{SU}% special unitary group
\DeclareMathOperator{\Gal}{Gal}%Galois  
\DeclareMathOperator{\Eis}{Eis}%Eisenstein

\DeclareMathOperator{\rank}{rank}%Rank
\DeclareMathOperator{\Tr}{Tr}%trace
\DeclareMathOperator{\Res}{Res}%Restriction of scalars
\DeclareMathOperator{\Norm}{Norm}%Norm

\DeclareMathOperator{\diag}{diag}%diag
\DeclareMathOperator{\Vor}{Vor}% Voronoi
\DeclareMathOperator{\Stab}{Stab}%Stabilizer

\newcommand{\shapecolor}{(shape$\mid$color)}
\def\BorelSerre{\text{BS}}
\def\Eis{\text{Eis}}

\theoremstyle{plain}
\newtheorem{thm}{Theorem}[section]
\newtheorem{lem}[thm]{Lemma}

\newtheorem{prop}[thm]{Proposition}
\newtheorem{conj}[thm]{Conjecture}

\theoremstyle{definition}

\begin{document}

\title[On the growth of torsion in cohomology of arithmetic groups]{On the growth of torsion in
the cohomology of arithmetic groups}

\author[Ash]{A. Ash}
\address{Boston College\\ Chestnut Hill, MA 02467}
\email{Avner.Ash@bc.edu}

\author[Gunnells]{P. E. Gunnells}
\address{Department of Mathematics and Statistics\\University of
Massachusetts\\Amherst, MA 01003-9305}
\email{gunnells@math.umass.edu}

\author[McConnell]{M. McConnell}
\address{Princeton University\\ Princeton, New Jersey 08540}
\email{markwm@princeton.edu}

\author[Yasaki]{D. Yasaki}
\address{Department of Mathematics and Statistics\\ 
University of North Carolina at Greensboro\\Greensboro, NC 27412}
\email{d\_yasaki@uncg.edu}

\renewcommand{\setminus}{\smallsetminus}

\date{\today} 

\thanks{The authors thank the Banff International Research Station and
Wesleyan University, where some work was carried out on this paper.
AA was partially supported by NSA grant
H98230-09-1-0050. PG was partially supported by NSF grants DMS 1101640 and 1501832. DY
was partially supported by NSA grant H98230-15-1-0228. This manuscript is submitted for
publication with the understanding that the United States government is
authorized to produce and distribute reprints.}

\keywords{Cohomology of arithmetic groups, Galois
representations, torsion in cohomology}

\subjclass[2010]{Primary 11F75; Secondary 11F67, 11F80, 11Y99}

\begin{abstract}
Let $G$ be a semisimple Lie group with associated symmetric space $D$,
and let $\Gamma \subset G$ be a cocompact arithmetic group.  Let $\sL$
be a lattice inside a $\ZZ \Gamma$-module arising from a rational
finite-dimensional complex representation of $G$.  Bergeron and
Venkatesh recently gave a precise conjecture about the growth of the
order of the torsion subgroup $H_{i} (\Gamma_{k}; \sL )_{\tors}$ as
$\Gamma_{k}$ ranges over a tower of congruence subgroups of $\Gamma$.
In particular they conjectured that the ratio ${\log |H_{i}
(\Gamma_{k} ; \sL)_{\tors}|}/{[\Gamma : \Gamma_{k}]}$ should tend to a
nonzero limit if and only if $i= (\dim(D)-1)/2$ and $G$ is a group of
deficiency $1$.  Furthermore, they gave a precise expression for the
limit.  In this paper, we investigate computationally the cohomology
of several (non-cocompact) arithmetic groups, including $\GL_{n}
(\ZZ)$ for $n=3,4,5$ and $\GL_{2} (\OO)$ for various rings of
integers, and observe its growth as a function of level.  In all cases
where our dataset is sufficiently large, we observe excellent
agreement with the same limit as in the predictions of
Bergeron--Venkatesh.  Our data also prompts us to make two new
conjectures on the growth of torsion not covered by the
Bergeron--Venkatesh conjecture.
\end{abstract}

\maketitle

\section{Introduction}\label{s:intro}

\subsection{}
Let $\bG$ be a connected semisimple $\QQ$-group with group of real
points $G = \bG (\RR)$.  Suppose $\Gamma \subset \bG (\QQ)$ is an
arithmetic subgroup and $\sM$ is a $\ZZ \Gamma$-module arising from a
rational finite-dimensional complex representation of $G$.  The
cohomology spaces $H^{*} (\Gamma ; \sM)$ are important objects in
number theory.  By a theorem of Franke \cite{franke} they can be
computed in terms of certain automorphic forms.  Moreover the
Langlands philosophy predicts connections between these automorphic
forms and arithmetic geometry (counting points mod $p$ of algebraic
varieties, Galois representations, and so on).

Now suppose that $\sM$ has an integral structure, i.e., a lattice $\sL
\subset \sM$ such that $\sM = \sL \otimes \CC$ and $\Gamma \sL \subset
\sL$.  Then one can consider the groups $H^{*} (\Gamma ; \sL)$; for
each $i$ the group $H^{i} (\Gamma ; \sL)$ is a finitely generated
abelian group, and thus has a torsion subgroup $H^{i} (\Gamma
;\sL)_{\tors}$.  In the past 30 years it has become understood that
torsion classes, even when they do not arise as the reduction mod $p$
of a characteristic 0 class, should also be connected to arithmetic.
Indeed, already in the 1980s Elstrodt--Gr\"unewald--Mennicke
\cite{egm.exeter} observed relationships between Hecke eigenclasses in
the torsion of abelianizations of congruence subgroups of $\PSL_{2}
(\ZZ [\sqrt{-1}])$ and the arithmetic of Galois extensions of $\QQ
(\sqrt{-1})$.  Later one of us (AA) conjectured that any Hecke
eigenclass $\xi \in H^{*} (\Gamma , \FF_{p})$, $\Gamma \subset \SL_{n}
(\ZZ)$, should be attached to a Galois representation $\rho \colon
\Gal(\Qbar/\QQ)\longrightarrow \GL_{n} (\FF_{p})$, in the sense that
for almost all primes $l$ the characteristic polynomial of the
Frobenius conjugacy class should equal a certain polynomial
constructed from the $l$-Hecke eigenvalues of $\xi$ \cite{ash.92}.
This is now a theorem due to P.~Scholze \cite{scholze} (conditional on
stabilization of the twisted trace formula).  The torsion in the
cohomology, even if it does not arise as the reduction of
characteristic $0$ classes and thus a priori has no connection with
automorphic forms, does in fact play a significant role when one
studies connections between cohomology of arithmetic groups and
arithmetic.

\subsection{}
In light of this, it is natural to ask what kind of torsion one
expects in the cohomology of a given arithmetic group.  For instance,
one might ask for which $\bG$ can one expect to find arithmetic
subgroups $\Gamma$ with $H^{*} (\Gamma ; \sL)_{\tors}$ large?  If one
expects torsion in the cohomology, which degrees should be
interesting, and how should the torsion grow as the index of $\Gamma$
increases?  In fact, it has been known for a long time that some
arithmetic groups have little or no torsion in cohomology, whereas
others have much.  For example, consider a torsionfree congruence
subgroup $\Gamma \subset \SL_{2} (\ZZ)$.  Since $\Gamma$ is
torsionfree, $H^{i} (\Gamma; \ZZ )$ vanishes unless $i=0,1$, and both
cohomology groups are torsionfree: we have $H^{*} (\Gamma ; \ZZ)
\simeq H^{*} (\Gamma \backslash \HH_{2}; \ZZ)$, where $\HH_{2}$ is the
upper halfplane, and $\Gamma \backslash \HH_{2}$ has a finite graph as
a deformation retract.  On the other hand, suppose $\Gamma$ is a
torsionfree congruence subgroup of $ \SL_{2} (\OO)$, where $\OO$ is
the ring of integers in an imaginary quadratic field.  Then $H^{i}
(\Gamma ; \ZZ )$ vanishes unless $i=0,1,2$, but now the torsion
behavior is quite different.  The group $H^{2}$, for instance,
typically contains lots of torsion (cf. \cite{priplata,
haluk.bianchi}).  In fact one observes what appears to be the initial
stage of exponential growth in its
torsion subgroup as the index of $\Gamma$ goes to infinity.  That
there should be exponential growth in the torsion of the homology of such
$\Gamma$ was already observed in the 1980s in unpublished computations
of F.~Gr\"unewald (of $H_{1} (\Gamma)$).

\subsection{}
When should there be lots of torsion in the cohomology of an
arithmetic group?  One answer is suggested by Bergeron--Venkatesh in
\cite{bv}, who formulated a precise conjecture for the growth of the
torsion in the \emph{homology} of $\Gamma$ when $\Gamma$ is
\emph{cocompact}.  To state it, we need more notation.  Recall that
$G=\bG (\RR)$ is the group of real points of our algebraic group
$\bG$.  Let $K\subset G$ be a maximal compact subgroup and let $D=G/K$
be the associated global symmetric space.  Let $\delta$ be the
\emph{deficiency} of $\bG$, defined by $\delta = \rank G - \rank K$,
where $\rank$ denotes the absolute rank (i.e., the rank over $\CC$).
The deficiency is an important invariant in the representation theory
of $\bG$.  For instance, $G$ has discrete series representations if
and only if $\delta =0$.  Then we have the following conjecture of
Bergeron--Venkatesh:

\begin{conj}[{\cite[Conjecture 1.3]{bv}}]\label{conj:bv}
Suppose $\bG$ has $\QQ$-rank $0$.  Let $\Gamma \supset \Gamma_{1}\supset \Gamma_{2}\supset
\dotsb$ be a decreasing family of (cocompact) congruence subgroups with $\cap_{k}
\Gamma_{k} = \{1 \}$.  Then
\[
\lim_{k\rightarrow \infty }\frac{\log |H_{i} (\Gamma_{k} ;
\sL)_{\tors}|}{[\Gamma : \Gamma_{k}]}
\]
exists for each $i$ and is zero unless $\delta = 1$ and $i= (d-1)/2$,
where $d=\dim D$.  In that case, the limit is strictly positive and
equals an explicit constant $c_{G,\sL}$ times $\mu (\Gamma)$, the
volume of $\Gamma \backslash D$.
\end{conj}

Thus the deficiency $\delta$, which depends only on the group of real
points of $G$ and not on its $\QQ$-structure, is the main quantity
that conjecturally determines whether or not an arithmetic group
$\Gamma$ should be expected to have lots of torsion.  Some heuristic
motivation for this phenomenon, which ultimately is inspired by
Bhargava's conjectures for the asymptotic counting of number fields
\cite{bhargava}, can be found in \cite[\S 6.5]{bv}.

In \cite{bv} the authors are able to obtain results in the direction
of Conjecture \ref{conj:bv} for certain $G$ and certain modules $\sL$.
Let us say that $\sL$ is \emph{strongly acyclic} for the family $\{\Gamma_{k}
\}$ if the spectra of the (differential form) Laplacian on $\sL
\otimes \CC$-valued $i$ forms on $\Gamma_{k}\backslash D$ are
uniformly bounded away from $0$ for all degrees $i$ and all
$\Gamma_{k}$.  This implies in particular that the homology $H_{i}
(\Gamma_{k}; \sL)$ is all torsion, i.e., the cohomology $H_{i}
(\Gamma_{k}; \sL\otimes \QQ)$ is trivial.  Such modules can always be
shown to exist for any $\Gamma$.  Then if $\delta =1$,
under these assumptions Bergeron--Venkatesh prove \cite[(1.4.2)]{bv}
an ``Euler characteristic'' version of Conjecture \ref{conj:bv}:
\begin{equation}\label{eq:bvthm}
\lim \inf_{k} \sum_{i} (-1)^{i+ (d-1)/2}\frac{\log |H_{i} (\Gamma_{k} ;
\sL)_{\tors}|}{[\Gamma : \Gamma_{k}]} = c_{G,\sL} \vol (\Gamma \backslash D).
\end{equation}
Further, they prove polynomial bounds on the torsion on $H_{0}$ and
$H_{d-1}$, which allows them to isolate the contribution of the
remaining homology groups in low-dimensional examples.  In particular
they show that Conjecture \ref{conj:bv} is true for $G=\SL_{2} (\CC)$
(again with the assumptions of cocompact quotients and strongly
acyclic coefficient modules).

\subsection{}
In this paper, we investigate the growth of torsion in cohomology when
the $\QQ$-rank of $\bG$ is nonzero and $\Gamma$ is not
cocompact. Moreover, we consider only the case of trivial
coefficients.  We look at these $\Gamma$ and these coefficients
because extensive computations become feasible for them.  Based on our
experimental evidence, in \S\ref{s:towers} we make two new
conjectures, similar to Conjecture \ref{conj:bv} but for all
arithmetic groups, cocompact or not.  

We treat a variety of examples with deficiencies $\delta =1$ and $2$:

\begin{itemize}
\item [$\delta =1$:] We consider $\Gamma \subset \GL_{n} (\ZZ)$ with $n=3,4$,
and $\Gamma \subset \GL_{2} (\OO_{F})$, where $F$ is the nonreal cubic field
of discriminant $-23$.  We also consider $\Gamma \subset  \GL_{2} ( \OO_{L})$
where $L$ is imaginary quadratic, which complements work of \sengun 
\cite{haluk.bianchi} and Pfaff \cite{pfaff.bianchi}.
\item [$\delta =2$:] We consider $\Gamma \subset \GL_{5} (\ZZ)$ and
$\Gamma \subset \GL_{2} (\OO_{E})$, where $E$ is the field of fifth
roots of unity.
\end{itemize}

We computed cohomology groups of congruence subgroups of these
arithmetic groups, with several questions in mind:

\begin{itemize}
\item Do we see distinct qualitative behavior in the growth in the
torsion in cohomology for $\delta =1$ and $\delta =2$?
\item How fast does the torsion appear to grow as the level (or norm
of the level) increases?  For example, do we see exponential growth
when $\delta =1$?  Does the analogue of Conjecture \ref{conj:bv} for
$\Gamma$ non-cocompact seem
to hold for cohomology?  What about the analogue of Bergeron--Venkatesh's ``Euler
characteristic'' theorem \eqref{eq:bvthm}?
\item When the torsion does appear to grow exponentially, how does the
prime factorization of the torsion order behave?  For instance, do we
see small prime divisors with large exponents, or large primes with
small exponents?
\item If the torsion does not appear to grow exponentially, what
behavior do we see?  For example, do the prime divisors of the torsion
appear to get arbitrarily large?
\item How does the torsion behave in different cohomological degrees?
For instance, is one degree singled out, as in Conjecture
\ref{conj:bv}, or do similar large amounts of torsion appear in other
degrees?  If a group has a cuspidal range, do we see different torsion
behavior across the range?   Does behavior outside the cuspidal range
differ? 
\item The Eisenstein cohomology of an arithmetic group is, roughly
speaking, the part of the complex cohomology that comes from lower
rank groups via the boundary of a compactification of $D$
\cite{harder.kyoto}.  Do we see Eisenstein cohomology phenomena
for torsion classes?
\end{itemize}

In light of our experimental evidence, we make two new conjectures,
along the lines of Conjecture~\ref{conj:bv} but more general.  We remove the
condition that the arithmetic groups in question be cocompact, and we
allow them to increase in level without necessarily being arranged in
a tower.  One conjecture is for groups of prime level and the other
groups of all levels.  The exact conjectures may be found in
\S\ref{s:towers}. 

\subsection{}
We now give a guide to the paper.  In \S\ref{s:background} we describe
the tools we use to compute cohomology.  In \S\ref{s:delta1} we
explicitly compute the constant $c_{G,\sL}$ appearing in Conjecture
\ref{conj:bv} for the $G$ of deficiency 1 we consider.  Section
\ref{s:overview} gives an overview of the scope of our computations
and gives the basic plots of our data.  \iffloattoend (All figures can
be found at the end of the paper.)  \fi The next sections present
analysis of our data: \S\ref{s:cuspidalrange} discusses the variation
of behavior of the torsion in various cohomological degrees for the
same group, \S \ref{s:chi} discusses the ``Euler
characteristic'' version of Conjecture \ref{conj:bv}, \S
\ref{s:towers} discusses towers of congruence subgroups and makes two
conjectures concerning the growth of their torsion, and
\S\ref{s:eisenstein} discusses torsion in Eisenstein cohomology.  We 
present tentative conclusions in \S \ref{s:conclusions}.

\section{Background and computational tools}\label{s:background}

\subsection{}
Let $F$ be a number field of degree $r + 2s$, where $r$ is the
number of real places and $s$ is the number of complex places.  Let
$\OO \subset F$ be the ring of integers of $F$.  
Let $\bG$ be the algebraic group $\Res_{F/\QQ}\GL_n$, and let $G =
\bG(\RR)$, the group of real points of $\bG$.  Then $G \simeq \prod_v
\GL_n(F_v)$, where the product is over the real and complex
embeddings~$v$ of $F$.  Let $K \subset G$ be a maximal compact
subgroup, and let  
$A_G$ be the split component of $\bG$, i.e., the identity component of
the real points of the maximal $\QQ$-split torus in the center of
$\bG$.  Let $D = G/A_GK$.  Then $D$ is a \emph{type $S-\QQ$
homogeneous space for $G$}, in the terminology of \cite{BS}.

In particular, $D$ is the Riemannian symmetric space for the Lie group
$G/A_G$.  Any congruence subgroup $\Gamma \subset \GL_n(\OO)$ acts on
$D$ by left multiplication.  Note that, since $\bG$ is reductive
rather than semisimple, the space $D$ is not necessarily a product of
irreducible symmetric spaces of nonpositive curvature: in general, in
addition to these there will be Euclidean factors that account for the
nontrivial units in $\OO$.

\subsection{}
We now introduce a model for $D$ that is more amenable to our
computational techniques.  Let $V_\RR$ be the $n(n + 1)/2$-dimensional
$\RR$-vector space of $n \times n$ symmetric matrices with real
entries, and let $V_\CC$ be the $n^2$-dimensional $\RR$-vector space
of $n \times n$ Hermitian matrices with complex entries.  Let $C_\RR
\subset V_\RR$, respectively $C_\CC \subset V_\CC$, denote the
codimension 0 open cone of positive definite matrices.  Let $V =
\prod_v V_v$, where
   \[V_v =  \begin{cases} V_\RR & \text{if $v$ is real, and}\\
V_\CC & \text{if $v$ is complex.}\\
 \end{cases}
\]
Define $C_v$ analogously.  
There is a left action of the real group $\GL_n(F_v)$ on
$C_v$ defined by 
\begin{equation} \label{eq:G-action}
g \cdot Q = g Q g^*,
\end{equation}
for each $g \in \GL_n(F_v)$ and $Q \in C_v$, where ${}^*$ is transpose
if $v$ is real and complex conjugate transpose if $v$ is complex.
Since $G \simeq \prod_v \GL_n(F_v)$, equation \eqref{eq:G-action}
gives an action of $G$ on $V$.  With this action, $A_G \simeq
\RR_{>0}$ acts by homotheties; that is, $h \in A_G$ acts on $V$ by
simultaneous scaling in each factor.  This action preserves $C = \prod
C_v$, and exhibits $G$ as the full automorphism group of $C$.  Thus we
get an identification $D \simeq C/\RR_{> 0}$.  A straightforward
computation shows that $D$ is a real manifold of dimension
$d = ((r+2s) n^2 + rn - 2)/2$.

There is a map $q \colon \OO^n \to V$ defined by 
\begin{equation}
  (q(x))_v = x_v x_v^*. 
\end{equation}
The rays defined by $q(x)$ for $x \in \OO^n$ give us a notion of
\emph{cusps} for $D$.

Recall that a \emph{polyhedral cone} in a real vector space $V$ is a
subset $\sigma$ which has the form 
\[\sigma = \biggl\{\sum_{i = 1}^p \lambda_i v_i \biggm | \lambda_i \geq
  0\biggr\},\] where $v_1, \dots, v_p$ are vectors in $V$.  In this case, we
say that $v_1, \dots, v_p$ span $\sigma$.  We are most interested in
polyhedral cones where each $v_i$ has the form $q(w_i)$ for some $w_i
\in \OO^n$, and for the remainder of the paper, we will use the term
polyhedral cone to include this additional condition.\footnote{We
remark that at this point one would typically expect to see a
rationality condition imposed on polyhedral cones.  Our additional
condition is an analogue of this, but it is not the same.  Indeed, for
general $F$ the $\ZZ$-span of the points $\{q (w)\mid w\in \OO^{n} \}$
will not be a lattice in $V$.}

\subsection{}
An explicit reduction theory due to Koecher \cite{koecher},
generalizing Voronoi's theory of perfect quadratic forms over $\QQ$
\cite{voronoi1}, gives a tessellation of $C$ by convex polyhedral
cones, which we make precise below.  Recall that a set $\Sigma$ of polyhedral
cones in a vector space forms a \emph{fan} if (i) $\sigma \in \Sigma$
and $\tau \subset \sigma $ implies $\tau \in \Sigma$ and (ii) if
$\sigma, \sigma '\in \Sigma$, then $\sigma \cap \sigma '$ is a common
face of $\sigma ,\sigma '$.

\begin{thm}[\cite{koecher}]
  There is a fan $\tilde{\Sigma}$ in $V$ with $\Gamma$-action such
  that the following hold:
  \begin{enumerate}
  \item There are only finitely many $\Gamma$-orbits in $\tilde{\Sigma}$.
  \item Every $y \in C$ is contained in the interior of a unique cone in
    $\tilde{\Sigma}$.
  \item Given any cone $\sigma \in \tilde{\Sigma}$ with $\sigma \cap C \neq
    \emptyset$, the stabilizer $\Stab(\sigma) = \set{\gamma \in \Gamma
      \mid \gamma \cdot \sigma = \sigma}$ is finite. 
  \end{enumerate}
\end{thm}

The structure of $\tilde{\Sigma}$ can be computed explicitly
using a generalization of Voronoi's algorithm for enumeration of
perfect quadratic forms.  For details we refer to \cite[\S
3]{gy-neg23} and the references given there; here we only sketch the
situation.

Each point of $C$ corresponds to a positive definite quadratic form on
$\OO^n$, given by summing the symmetric and Hermitian forms
corresponding to the real and complex places of $F$.  A form is called
\emph{perfect} if it can be reconstructed from the knowledge of its
minimal nonzero value on $\OO^{n}$ and of the vectors on which it attains
this minimum.  Then it is known that the perfect forms are in
bijection with cones in $\tilde{\Sigma}$ of maximal dimension, which
are therefore called \emph{perfect pyramids}.  Namely, a given perfect
pyramid is generated as a convex cone by $q(w_1), \dots, q(w_p)$,
where $w_1, \dots, w_p$ are the minimal vectors of the corresponding
perfect form.  Thus computing the structure of $\tilde{\Sigma}$
reduces to a classification problem in quadratic forms.  Voronoi
described an algorithm to solve this problem for $F=\QQ$ that
immediately generalizes to general $F$ in this setting.

Modulo homotheties, the fan $\tilde{\Sigma}$ descends to a
$\GL_n(\OO)$-stable tessellation of $D$ by ``ideal'' polytopes.
Specifically, let $C^{*}$ be the union of the closures of the perfect
pyramids in the cone $\overline{C}$ of positive semi-definite matrices
(minus the origin). The $(k + 1)$-dimensional cones in
$\tilde{\Sigma}$ that intersect non-trivially with $C$ descend to give
$k$-dimensional cells in $D$.  Let $\Sigma^*_k(\Gamma)$ denote a set
of representatives, modulo the action of $\Gamma$, of these
$k$-dimensional cells, and let $\Sigma^*(\Gamma) = \cup
\,\Sigma^*_k(\Gamma)$.

For example, when $n = 2$ and $F = \QQ$, the space $D$ can be
identified with the complex upper halfplane $\HH_{2}$.  Under this
identification, the cusps are precisely $\PP^1(\QQ)$ as expected. The
Voronoi tessellation corresponds to the Farey tessellation
\cite{hurwitz-qf} of $\HH_{2}$ defined by $\GL_2(\ZZ)$-translates of the
ideal triangle with vertices $\set{0,1,\infty}$.  When $n=2$ and
$F=\QQ (i)$, $i^{2} = 1$, the space $D$ can be identified with
hyperbolic $3$-space $\HH_{3}$.  The cusps can be identified with
$\PP^{1} (\QQ (i))$.  The cells in the tessellation are ideal
octahedra \cite{Crhyp, egm.exeter,grun.book} that are translates of
the convex hull of $\{0,1,i,i+1,(i+1)/2,\infty \}$.

\subsection{}
We use the cells $\Sigma^*$ to define a chain complex $\Vor(\Gamma) =
(V_*, d_*)$, which we call the \emph{Voronoi--Koecher complex}; we
refer to \cite{gy-neg23, zeta5,PerfFormModGrp,gs,aimpaper,AGM} for
details and further examples.  Over $\CC$, the homology of the
Voronoi--Koecher complex is isomorphic to $H^{*} (\Gamma ;
\tilde{\Omega}_{\CC })$, where $\tilde{\Omega}_{\CC}$ is the local
coefficient system attached to $\Omega \otimes \CC$, where $\Omega$ is
the orientation module of $\Gamma$.  Over $\ZZ$, the homology of the
Voronoi--Koecher complex is not quite the integral cohomology of
$\Gamma$ (again with twisted coefficients corresponding to the
orientation module), but does agree with the integral cohomology
modulo certain small primes.  More precisely, recall that a prime is a
\emph{torsion prime} of $\Gamma $ if it divides the order of a torsion
element of $\Gamma $.  For any positive integer $n$ let $\cS_{n}$ be
the Serre class of finite abelian groups with orders only divisible by
primes less than or equal to $n$.  Then we have the following theorem,
which slightly generalizes results of \cite[\S 3]{aimpaper}.
The proof given there holds in this more general setting.
\begin{thm}\label{th:voronoiiso}
Let $b$ be an upper bound on the torsion primes for $\Gamma$.  Then,
modulo the Serre class $\cS_b$, we have
\begin{equation}\label{eq:voronoiiso}
H^k(\Gamma; \tilde{\Omega}_{\ZZ }) \simeq H_{d -
k}(\Vor(\Gamma); \ZZ )\end{equation}
for all $k$.  
\end{thm}

Hence, if one ignores small primes, and is willing to twist
coefficients by the orientation module, the Voronoi--Koecher complex
provides a tool to investigate torsion in the cohomology of $\Gamma$.  
The following theorem tabulates the torsion primes for the examples we
consider.  We omit the proof.

\begin{thm}
Let $L$ be imaginary quadratic, let $F$ be the complex cubic field of
discriminant $-23$, and let $E$ be the field of fifth roots of unity.
The deficiency $\delta$, dimension $d$ of the symmetric space $D$, and the
torsion primes for the arithmetic groups $\Gamma$ we consider in this
paper are as follows:
\[
\begin{array}{|l|c|r|r|}
  \hline
\Gamma & \delta & d & \text{torsion primes}\\
\hline
\hline
\GL_2(\OO_L) & 1 & 3 & 2,3 \\ \hline
\GL_3(\ZZ) & 1 & 5  & 2,3\\ \hline
\GL_2(\OO_F) & 1 & 6 & 2,3 \\ \hline
\GL_4(\ZZ) & 1 & 9  & 2,3,5\\ \hline
\GL_2(\OO_E) & 2 & 7 &  2,3,5 \\ \hline
\GL_5(\ZZ) & 2 & 14 &  2,3,5\\ \hline
\end{array}
\]
\end{thm}

\subsection{}
There is one additional class of primes that should be mentioned here,
since they are a potential source of torsion in the cohomology that is
in some sense trivial: these torsion classes arise from the congruence
covers of $\Gamma$.  For more details, we refer to \cite[\S
3.7]{cv.torsionjl}.  For any arithmetic group $\Gamma$, let $\hat
\Gamma$ be its congruence completion, that is the completion of
$\Gamma $ for the topology defined by congruence subgroups.  The map
$\Gamma \rightarrow \hat \Gamma$ induces a map in cohomology $H^{*}
(\hat \Gamma ) \rightarrow H^{*} (\Gamma)$, and the congruence
cohomology is the image.  From a practical point of view, when the
group is $\Gamma_{0} (\fn)$, the relevant primes are those dividing
$\Norm (\fp)-1$, as $\fp$ ranges over prime ideals dividing $\fn$.  We
call any such prime a \emph{congruence prime}.  Any prime that is not
a torsion prime or a congruence prime will be called an \emph{exotic
prime}.

\section{The conjectured limit in the case of deficiency
\texorpdfstring{$1$}{1}}\label{s:delta1} 

\subsection{}
Let us call the constant $c_{G,\sL} \mu (\Gamma)$ appearing in
Conjecture \ref{conj:bv} the \emph{B-V limit}.  In this section we
compute the B-V limits for the groups of deficiency 1 that we
consider.  Recall that Conjecture \ref{conj:bv} states that if $G$ is
semisimple and $\Gamma \supset \Gamma_{1}\supset \Gamma_{2}\supset
\dotsb$ is a decreasing family of cocompact congruence subgroups with
$\cap_{k} \Gamma_{k} = \{1 \}$, then
\[
\lim_{k\rightarrow \infty }{\log |H_{i} (\Gamma_{k} ;
\sL)_{\tors}|}/{[\Gamma : \Gamma_{k}]}
\]
should exist and vanish in all cases except $i= (d-1)/2$, in which
case it should equal the constant $c_{G,\sL}\mu (\Gamma)$, where
$\mu(\Gamma) = \vol(\Gamma \backslash D)$.  According to \cite[\S\S
3--5]{bv}, the constant $c_{G,\sL}$ is the $L^{2}$-torsion
$t^{(2)}_D(\rho)$, where $\rho$ is the rational representation giving
rise to the local system $\sL$.  In our cases, we need to incorporate
several features in our discussion.

First, the groups we work with are reductive rather than semisimple.
This makes little conceptual difference, in that we simply have to be
cognizant of how $\GL$ vs.~$\SL$ affects the volume computation $\mu
(\Gamma)$.  For most cases this is not difficult, except when the unit
group $\OO^{\times}$ has nontrivial rank.  In this case the symmetric
space for $\GL$ has additional flat factors to accommodate the action
of the units.  These clearly have little arithmetic significance, but
we must consider them when computing $\mu (\Gamma)$.

Second, since we work with $\GL$ instead of $\SL$, we must incorporate
the orientation module $\tilde{\Omega}_{\ZZ}$ into the discussion when
we compare the homology of the Voronoi complex with the group homology
(Theorem \ref{th:voronoiiso}).  This extra local system
$\tilde{\Omega}_{\ZZ}$ has no arithmetic significance and is an
artifact of our computational technique.  In fact, if we worked
instead with the subgroup $\Gamma '\subset \GL_{n}(\OO) $ of elements
with totally positive determinant, then the orientation module
restricts to the trivial representation.  Since the index $[\GL_{n}
(\OO) : \Gamma ']$ is a $2$-power, only the $2$-torsion in the
homology is affected.  Thus we will assume that the relevant
representation $\rho$ is the trivial representation when computing the
analytic torsion.

\subsection{}
We compute this value for the examples considered in this paper after
some preliminary results.

\begin{lem}\label{lem:index}
Let $n \geq 1$ and let $\OO$ be the ring of integers in a number field
$k$.  Then 
\[[\PGL_n(\OO): \PSL_n(\OO)] = [\OO^* : (\OO^*)^n].\]
\end{lem}
\begin{proof}
We have the following commutative diagram, where $\eta _{n}$ is the
(possibly trivial) subgroup of all $a\in \OO^{*}$ satisfying
$a^{n}=1$, and $(\OO^*)^n$ is the subgroup of $n$th powers.  In the
diagram we realize $\eta_{n}$, $\OO^{*}$, and $(\OO^{*})^{n}$ as
subgroups of the matrix groups via the scalar matrices.  One can check
that the rows and columns are exact, where all of the maps are the
obvious ones.

\begin{center}
\leavevmode
\xymatrix{
&1\ar[d] &1\ar[d] &1\ar[d]& \\
1\ar[r]&\eta_{n}\ar[r]\ar[d]&\SL_{n} (\OO)\ar[r] \ar[d]& \PSL_n(\OO) \ar[d]\ar[r] & 1\\
1 \ar[r] & \OO^* \ar[r] \ar[d]^{\det}&\GL_n(\OO) \ar[r] \ar[d]^{\det} & \PGL_n(\OO)\ar[d]^{\det} \ar[r]& 1\\
1 \ar[r] & (\OO^*)^n \ar[r] \ar[d] &\OO^* \ar[r]\ar[d]& \OO^*/(\OO^*)^n \ar[d] \ar[r]& 1\\
{}  & 1  & 1  & 1  & {}
}
\end{center}

It follows that 
\[[\PGL_n(\OO): \PSL_n(\OO)] = |\OO^*/(\OO^*)^n| = [\OO^* : (\OO^*)^n].\]
\end{proof}
In particular, if $\OO$ is the ring of integers in an imaginary
quadratic field, a short computation together with Lemma
\ref{lem:index} implies
\begin{equation}\label{eq:imquadindex}
[\PGL_2(\OO): \PSL_2(\OO)] = 2.
\end{equation}

\subsection{}
Now we consider the volumes of our locally symmetric spaces.  We have
the following general result of Borel \cite{borel-commensurability}
that works for any number field $k$; in the special case of $k$
imaginary quadratic, this result goes back to Humbert.  Here
\[
\zeta_{k} (s)  = \sum_{I\subseteq \OO} \Norm (I)^{-s}
\]
denotes the Dedekind zeta function of $k$.  Let $\HH_2$ and $\HH_3$
denote hyperbolic 2-space and 3-space, respectively.
\begin{thm}[{\cite[Theorem~7.3]{borel-commensurability}}] \label{thm:borel-vol}
Let $k$ be a number field of degree $d = r + 2s$ and discriminant
$\Delta$ with ring of integers $\OO$.  Let $\HH = \HH_2^{r} \times
\HH_3^{s}$.  Then 
\[\vol(\PSL_2(\OO)\backslash \HH) = 2^{1 - 3s} \pi^{-d} |\Delta|^{3/2}
\zeta_k(2).\]
\end{thm}

We are now ready to compute the B-V limits attached to the
torsion growth.  We begin with imaginary quadratic fields.

\begin{prop}\label{prop:imquadconst}
Let $L$ be an imaginary quadratic field of discriminant $\Delta<0$ and
$\zeta_{L} (s)$ its Dedekind zeta function.  Then
\[\torconst\mu(\GL_2(\OO_L)) = \frac{|\Delta|^{3/2}}{48\pi^3} \zeta_L(2).\]
\end{prop}
\begin{proof}
 Since $L$ is imaginary quadratic, \cite[\S 5.9.3, Example 1]{bv}
shows $\torconst = \frac{1}{6\pi}$.  It remains to compute
$\mu(\GL_2(\OO_L))$.  In the notation of Theorem~\ref{thm:borel-vol},
we have $d = 2$, $s = 1$, and $\HH = \HH_3$.  Since ${\pm I}$ acts
trivially on $\HH $, Theorem~\ref{thm:borel-vol} implies
\[
\mu(\SL_2(\OO_L)) = \frac{|\Delta|^{3/2}}{4\pi^2} \zeta_L(2).
\]
Note that the center of $\GL_2(\OO_L)$ acts trivially on $\HH $.  It
follows that to compute $\mu(\GL_2(\OO_L))$ we need to divide not by 
the index $[\GL_2(\OO_L):\SL_2(\OO_L)]$, but rather the index
$[\PGL_2(\OO_L): \PSL_2(\OO_L)]$.  From
\eqref{eq:imquadindex}, this index is independent of
$\Delta$ and equals $2$.
Therefore
\[\mu(\GL_2(\OO_L)) = \frac{1}{2} \mu(\SL_2(\OO_L)) =
\frac{|\Delta|^{3/2}}{8\pi^2} \zeta_L(2),\]     
and multiplying by $\torconst$ gives the desired result.
\end{proof}

\subsection{}
Now we consider our cubic field of mixed signature:

\begin{prop}\label{prop:neg23vol}
Let $F$ be a nonreal cubic field of discriminant $\Delta<0$.  Then 
\begin{equation}\label{eq:cubictorconst}
\torconst \mu(\GL_2(\OO_F)) =  \frac{|\Delta|^{3/2} \reg_F}{48 \pi^5}
\zeta_F(2), 
\end{equation}
where $\reg_{F}$ denotes the regulator of $F$.
\end{prop}
\begin{proof}
The global symmetric space $D$ for $\GL_2(\OO_F)$ is $\HH_2 \times
\HH_3 \times \RR$, where the flat factor accounts for the nontrivial
units in $\OO_F$.  We have $\torconst = \frac{1}{2\pi}\cdot
\frac{1}{6\pi} = \frac{1}{12\pi^2}$.  Since ${\pm I}$ acts trivially
on $\HH_2 \times \HH_3$, Theorem~\ref{thm:borel-vol} implies
\begin{equation}\label{eq:nonrealvolumeSL}
\vol(\SL_2(\OO)\backslash (\HH_2 \times \HH_3)) =
\frac{|\Delta|^{3/2}}{4\pi^3} \zeta_F(2).
\end{equation}
In this case, the center does not act trivially on $D$, but instead
acts on the flat factor.  Thus we need to include the volume of the
flat factor modulo the action of $\OO^{\times}$.  We normalize so that
this volume is the regulator $\reg_{F}$, and we multiply
\eqref{eq:nonrealvolumeSL} by $\reg_F$ to pass to
$\mu(\GL_2(\OO))$. Finally, multiplying by $\torconst$ gives the
desired result.
\end{proof}
For our particular cubic field $F$ of discriminant $-23$, we have
\begin{align*}
\zeta_{F} (2) &= 1.110001006025\dotsc ,\\
\reg_{F} &=      0.281199574322\dotsc .
\end{align*}
Plugging into \eqref{eq:cubictorconst}, we find 
\[
\torconst \mu(\GL_2(\OO_F)) = 0.002343900569\dotsc .
\]

\subsection{}
Next we consider $\SL_{n} (\RR)$ and $\GL_{n} (\RR)$.

\begin{prop}[\cite{bv}] \label{prop:glntorconst}
The $L^2$-analytic torsion $\torconst$ for the trivial representation
of $\SL_n(\RR)$ is 
\[\torconst = R_n\frac{\pi\vol(\SO(n))}{\vol(\SU(n))},\] 
where $R_3 = \frac{1}{2}$, $R_4 =
\frac{124}{45}$, and $R_n = 0$ 
otherwise. 
\end{prop}
\begin{proof}
From \cite[Proposition~5.2]{bv}, $R_n = 0$ for $n \neq 3,4$.  Bergeron
and Venkatesh work out an explicit formula \cite[\S5.9.2]{bv} for the
$L^2$-analytic torsion $t^{(2)}_D(\rho)$ for a representation $\rho$
of $\SL_3(\RR)$ in terms of the highest weight $\lambda = p \epsilon_1
+ q \epsilon_2 + r \epsilon_3$.  Let
\begin{gather*}
A_1 = \frac{1}{2}(p + 1 - q), \quad A_2 =  \frac{1}{2}(p -r + 2),
\quad A_3 =  \frac{1}{2}(q - r + 1)\\
C_1 = \frac{1}{3}(p+q-2r+3), \quad C_2 = \frac{1}{3}(p+r-2q), \quad C_3 = \frac{1}{3}(2p-q-r+3).
\end{gather*}
Then 
\begin{equation}\label{eq:sl3tor}
t^{(2)}_D(\rho) = \frac{\pi \vol(\SO(n))}{\vol(\SU(n))}\left(2A_1A_3C_1C_3 +
2A_2|C_2|
\begin{cases}
  A_3C_3 &\text{if $C_2 \geq 0$,}\\
  A_1C_1 &\text{if $C_2 \leq 0$.}
\end{cases}\right)
\end{equation}
For the trivial representation, we have $A_1 = A_3 = \frac{1}{2}$,
$A_2 = C_1 = C_3 = 1$, and $C_2 = 0$.  Then \eqref{eq:sl3tor} gives 
\[\torconst = \frac{\pi\vol(\SO(3))}{2\vol(\SU(3))}\] as desired.

\subsection{}
Next we consider $n = 4$.  From \cite[Proposition~5.2, eq.~(5.9.6)]{bv}, we can express
$\torconst$ as \[\torconst = c(D) R_4,\] 
where $R_4$ is  equal up to sign to
a sum of integrals, and 
\[
c(D) = \frac{\pi \vol(\SO(4))}{\vol(\SU(4))}.
\]  
From \cite[\S5.8]{bv}, the constant $R_4$
depends only on the inner form of the Lie algebra.  Since $\SL_4(\RR)$
is isogenous to $\SO(3,3)$, we may use the computations in
\cite[\S5.9.3 Example (1)]{bv} for $\SO(5,1)$ to get the constant.
The computations there show that 
\[R_4 = \frac{t^{(2)}_{\HH^5}(\triv)}{c(\HH^5)} =
\frac{\frac{31}{45\pi^2}}{\frac{2!}{8\pi^2}} = \frac{124}{45}.\] 
\end{proof}

\begin{prop}
  The conjectured limit for $\GL_3(\ZZ)$ and $\GL_4(\ZZ)$ is  
\begin{align*}
\torconst\mu(\GL_3(\ZZ)) &= \frac{\sqrt{3}}{288 \pi^2 }\zeta(3) \\
\torconst\mu(\GL_4(\ZZ)) &=  \frac{31 \sqrt{2} }{259200\pi^2} \zeta(3). \\
\end{align*}
\end{prop}
\begin{proof}
Note that the intersection of the center of $\SL_n(\RR)$ with
$\SO(n,\RR)$ is trivial if $n$ is odd and $\set{\pm I}$ if $n$ is
even.  It follows that 
\[\mu(\SL_n(\ZZ)) = 
\begin{cases}
\dfrac{\vol(\SL_n(\ZZ)\backslash \SL_n(\RR))}{\vol(\SO(n,\RR))} &
\text{if $n$ is odd,}\\[3ex]
\dfrac{2\vol(\SL_n(\ZZ)\backslash \SL_n(\RR))}{\vol(\SO(n,\RR))} &
\text{if $n$ is even.}\\
\end{cases}
\]

Next we pass to $\GL_n(\ZZ)$.  When $n$ is odd, $-I \in \GL_n(\ZZ)
\setminus \SL_n(\ZZ)$.  It is in the center of $\GL_n(\ZZ)$, so acts
trivially on the symmetric space.  Thus $\mu(\GL_n(\ZZ)) =
\mu(\SL_n(\ZZ))$ when $n$ is odd.  When $n$ is even,
$\diag(-1,1,1,...,1)\in \GL_n(\ZZ) \setminus \SL_n(\ZZ)$ is not
central.  It acts non-trivially on the symmetric space, thus dividing
the volume of the quotient in half. From Lemma~\ref{lem:index}, we
have the index $[\PGL_n(\ZZ) : \PSL_n(\ZZ)]$ is equal to $1$ if $n$ is
odd and equal to $2$ if $n$ is even so $\mu(\GL_n(\ZZ)) =
\frac{1}{2}\mu(\SL_n(\ZZ))$ when $n$ is even.  Therefore
\[
\mu(\GL_n(\ZZ)) = \dfrac{\vol(\SL_n(\ZZ)\backslash
  \SL_n(\RR))}{\vol(\SO(n,\RR))}  
\]
in either case.

Combining with Proposition~\ref{prop:glntorconst}, we have 
\[
\torconst \mu(\GL_n(\ZZ)) = \pi R_n \frac{\vol(\SL_n(\ZZ) \backslash
  \SL_n(\RR))}{\vol(\SU(n))}
\] 
A classical computation (using the volume forms induced from the
invariant forms $(X,Y)\mapsto \Tr (XY)$ on the Lie algebras) shows
\begin{align}
\vol(\SL_n(\ZZ)\backslash \SL_n(\RR)) &= 2^{(n - 1)/2} \zeta(2)
\zeta(3) \dotsm \zeta(n),
\\
\vol(\SU(n)) &= \frac{\sqrt{n}(2\pi)^{(n^2 + n - 2)/2}}{1!2!\dotsm (n
  - 1)!}.
\end{align}

It follows that 
\[
\torconst \mu(\GL_n(\ZZ)) = \pi R_n
\frac{\prod_{k = 2}^n (k-1)!\zeta(k)}{2^{(n^2 
  - 1)/2}\pi^{(n^2 + n -2)/2}\sqrt{n}}
\]
Since $R_3 = \frac{1}{2}$, $R_4 = \frac{124}{45}$, $\zeta(2) =
\frac{\pi^2}{6}$, and $\zeta(4) = \frac{\pi^4}{90}$, we
have\footnote{We remark that the value 0.000732\dots for $\GL_{3}
(\ZZ)$ here corrects the value in Example 2 of \cite[\S 5.9.3]{bv}.}
\begin{align*}
\torconst \mu(\GL_3(\ZZ)) 
&=  \frac{\sqrt{3}}{288 \pi^2}\zeta(3) \\
&= 0.000732476036628004814191244682033\dotsc ,\\
\torconst \mu(\GL_4(\ZZ)) &=  \frac{31
  \sqrt{2}}{259200\pi^2} 
\zeta(3)\\ 
&=0.0000205999884056288780742643411677\dotsc 
. 
\end{align*}
\end{proof}

\subsection{}
We conclude by summarizing all the specific values for the 
the B-V limit computed in this section.
\begin{thm}
Let $L$ be the imaginary quadratic field of discriminant $\Delta $.  Let $F$
be the nonreal cubic field of discriminant $-23$.  Then the
quantities $c_{G,\triv } \mu(\Gamma)$ for the groups of deficiency
$1$ we consider are given in the following table:
\renewcommand{\arraystretch}{1.5}
\[
\begin{array}{|c||c|c|c|c|}
\hline
\Gamma & \GL_2(\OO_L) & \GL_2(\OO_F) & \GL_3(\ZZ) & \GL_4(\ZZ)\\ 
\hline
\rule{0pt}{2em} c_{G,\triv}\mu(\Gamma) & \dfrac{|\Delta|^{3/2}}{48\pi^3} \zeta_L(2) & \dfrac{23^{3/2} \reg_F}{48 \pi^5} \zeta_F(2) & \dfrac{\sqrt{3}}{288\pi^2 }\zeta(3) & \dfrac{31 \sqrt{2} }{259200\pi^2} \zeta(3) \\ [10pt]
\hline
\end{array}
\]
\end{thm}

\section{Overview of the data and basic plots}\label{s:overview}

\subsection{}
In this section we present plots of the torsion data we generated.  We
use the following conventions in our computations and our plots:

\begin{itemize}
\item The homology or cohomology group data we present is always
indexed in terms of the Voronoi homology, that is, in terms of the
homological labelling in Theorem \ref{th:voronoiiso}.  In particular,
a label \texttt{Hk} on a graph indicates that this group is the $k$th
Voronoi homology group, and therefore corresponds to cycles supported
on the Voronoi cells in $D$ of dimension $k$ (i.e.~the cells that are
the images of the $(k+1)$-dimensional cones in $C$ modulo
homotheties).
\item When different groups in
different degrees appear in the same graph, the first group in the
list corresponds, on the cohomology side of \eqref{eq:voronoiiso}, to
the cohomology group in the virtual cohomological dimension (vcd).
Thus for example in Figure \ref{fig:allgroupsgl3z} (for subgroups of
$\GL_{3} (\ZZ )$) one sees groups labelled \texttt{H2}, \texttt{H3},
and \texttt{H4}.  Since the dimension of $\GL_{3} (\RR )/ \Orth (3)$
is $5$, these Voronoi homology groups correspond to the cohomology
groups $H^{3}$, $H^{2}$, and $H^{1}$, respectively.
\item Data points have the same
\shapecolor\footnote{Here one should read \emph{shape}
(respectively, \emph{color}) if one is reading a black and white
(resp., color) version of this paper.} based on where their group falls
relative to the vcd.  Thus for example the
points labelled \texttt{H2} in Figure \ref{gl3.2} for
$\GL_{3} (\ZZ)$ have the same \shapecolor\ as those in Figure
\ref{23.1} labelled \texttt{H1}.  Both of these, under the
isomorphism \eqref{eq:voronoiiso}, correspond to the cohomology groups
at the vcd.
\item In our computations, we always worked with congruence subgroups
of the form $\Gamma_{0} (\fn)$ for some ideal $\fn \subset \OO$, where
$\OO$ is the relevant ring of integers.  By definition, $\Gamma_{0}
(\fn)$ is the subgroup of $\GL_{n} (\OO)$ with bottom row congruent to
$(0,\dots ,0,*)$ mod $\fn$.
\item In all plots, the label \texttt{Index} refers to $[\GL_{n} (\OO)
: \Gamma_{0} (\fn)]$, the index of the congruence group in the full
group, and the label \texttt{Log torsion/Index} refers to the ratio
$(\log |H_{i} (\Gamma)_{\tors}|)/[\Gamma : \Gamma_{0} (\fn)]$.
\item Each series of computations was run ordered by the norm of the
level.  Some levels were skipped because jobs crashed (overall memory
usage became too large).  However, this happened infrequently, so our
computations are nearly complete in the ranges of levels we
considered.
\end{itemize}

We now make a remark about the cost of our computations.  The
largest computation, in terms of memory usage and time for a
computation that finished, was to compute $H_6$ for $\GL_4$ at level
$66$.  The boundary matrix from $6$ cells to $5$ cells was $150491
\times 256782$, and the boundary matrix from $7$ cells to $6$ cells was
$256782 \times 216270$.  The computation took approximately $586.59$
hours, with memory usage approximately $126$ GB.

\subsection{\texorpdfstring{$\GL_n (\ZZ )$}{GLn}}

We were able to compute the Voronoi homology for the following
groups/levels of congruence subgroups:

\medskip
\begin{center}
\begin{tabular}{|c||c|p{300pt}|}
\hline
Group&Deficiency $\delta $&Level of congruence subgroup\\
\hline
$\GL_{3} (\ZZ)$&1&$H_{2}$: $N\leq 641$, $H_{3}$: $N\leq 641$, $H_{4}$: $N\leq 659$\\
$\GL_{4} (\ZZ)$&1&$H_{3}$: $N\leq 119$, $H_{4}$: $N\leq 99$, $H_{5}$: $N\leq 61$, $H_{6}$: $N\leq 74$, $H_{7}$: $N\leq 80$\\
$\GL_{5} (\ZZ)$&2&$H_{6}$: $N\leq 29$\\
\hline
\end{tabular}
\end{center}
\medskip

As one can clearly see, the difficulty of the computation increases
substantially as $n$ increases.  For $\GL_{3}$ and $\GL_{4}$, we
computed a range of homology groups that included the cuspidal range
($H_{2}$---$H_{3}$ for $\GL_{3}$, $H_{4}$---$H_{5}$ for $\GL_{4}$).  For
$\GL_{5} (\ZZ)$, we only worked with $H_{6}$, which is the top of the
cuspidal range.  We highlight some of the
results in this range.   The torsion size is given in factored form
with exotic torsion in \textbf{bold}. 

\subsubsection{\texorpdfstring{$\GL_3 (\ZZ )$}{GL3(Z)}}
\begin{description}
  \item[$H_2$] We computed $H_2$ for $639$ levels less than or equal to
    $641$.  Of these, there is nontrivial exotic torsion at $457$
    levels.  The largest torsion group occurs at level $570$.  The
    largest exotic torsion occurs at this same level.  The
    torsion size is 
\begin{multline*}
2^{154} \cdot 3^{27} \cdot \mathbf{5^{2}} \cdot \mathbf{7^{6}} \cdot
\mathbf{11^{6}} \cdot \mathbf{17^{10}} \cdot \mathbf{37^{6}} \cdot
\mathbf{47^{2}} \cdot \mathbf{131^{2}} \cdot \mathbf{619^{6}} \cdot
\mathbf{3137^{6}} \cdot \mathbf{6113^{6}} \cdot \\ \mathbf{2723737^{2}}
\cdot \mathbf{242222857291^{2}} \cdot \\
\mathbf{278917146364629278585122304155523929101710815974757^{2}}.
\end{multline*}
For an explanation of why all the exponents appearing in the exotic torsion
are even, see \cite{ash-parity}.

The largest exotic prime occurs at level $638$.   The
    torsion size is
\begin{multline*}
2^{106} \cdot 3^{22} \cdot 5^{8} \cdot 7^{2} \cdot \mathbf{11^{2}}
\cdot \mathbf{31^{6}} \cdot \mathbf{8969^{6}} \cdot\\
\mathbf{153783531368731301629667842625718224527480871638219475779265644361^{2}}. 
\end{multline*}
\item[$H_3$] We computed $H_3$ for $640$ levels less than or equal to
  $641$.  Of these, there is nontrivial exotic torsion at $48$ levels.
  The largest torsion group occurs at level $599$.  The torsion, none of
  which is exotic, has size 
  $13 \cdot 23$. The largest exotic torsion occurs at level
  $625$.   The torsion, all of which is exotic, has size
  $\mathbf{5^3}$.  The largest exotic 
  prime is $23$, which first occurs at 
  level $529$.  The torsion size is $11 \cdot \mathbf{23}$.
\item[$H_4$] We computed $H_4$ for $658$ levels less than or equal to
  $659$.  There were no levels with exotic torsion.  The largest
  torsion group occurs at level $600$.  The torsion size is $2^{791}$.
\end{description}

\begin{figure}[htb]
\begin{center}
\subfigure[$H_{2}$\label{gl3.2}]{\includegraphics[scale=0.5]{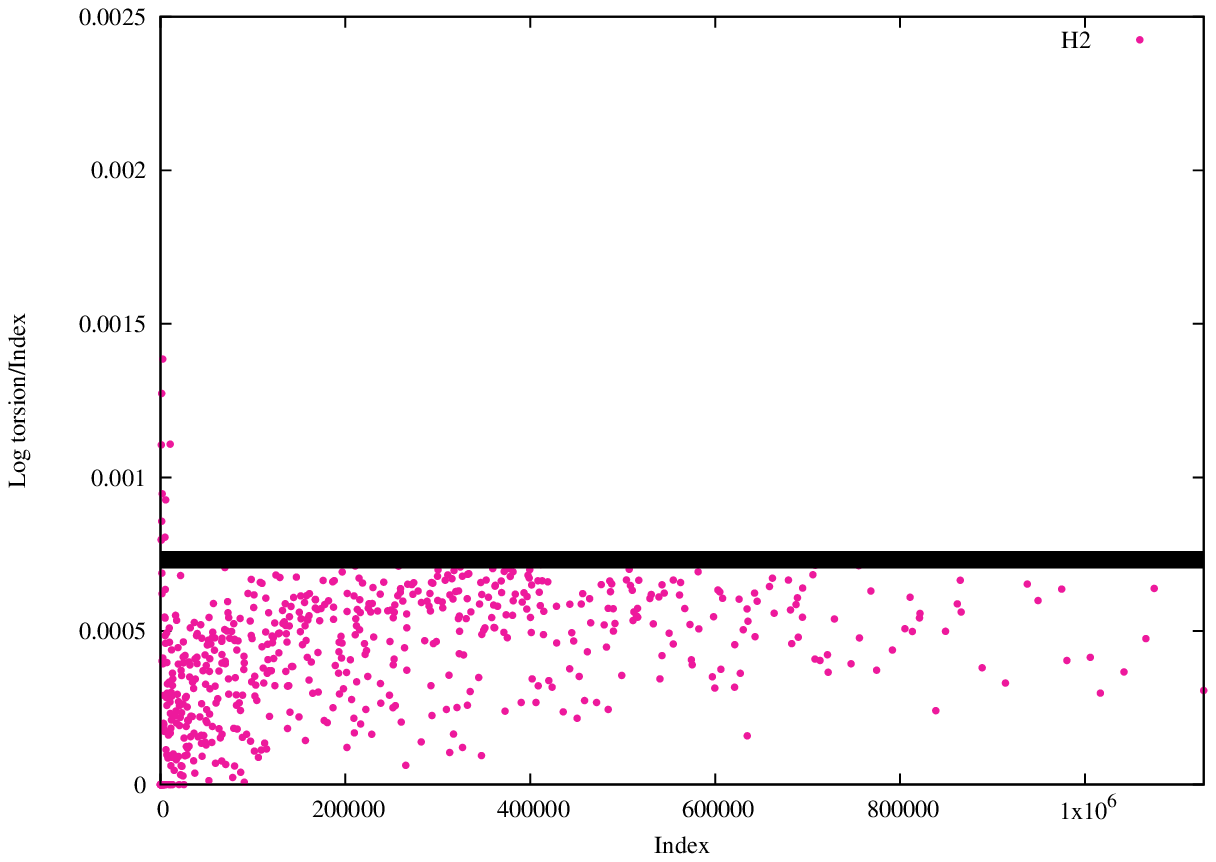}}
\quad\quad
\subfigure[$H_{3}$\label{gl3.3}]{\includegraphics[scale=0.5]{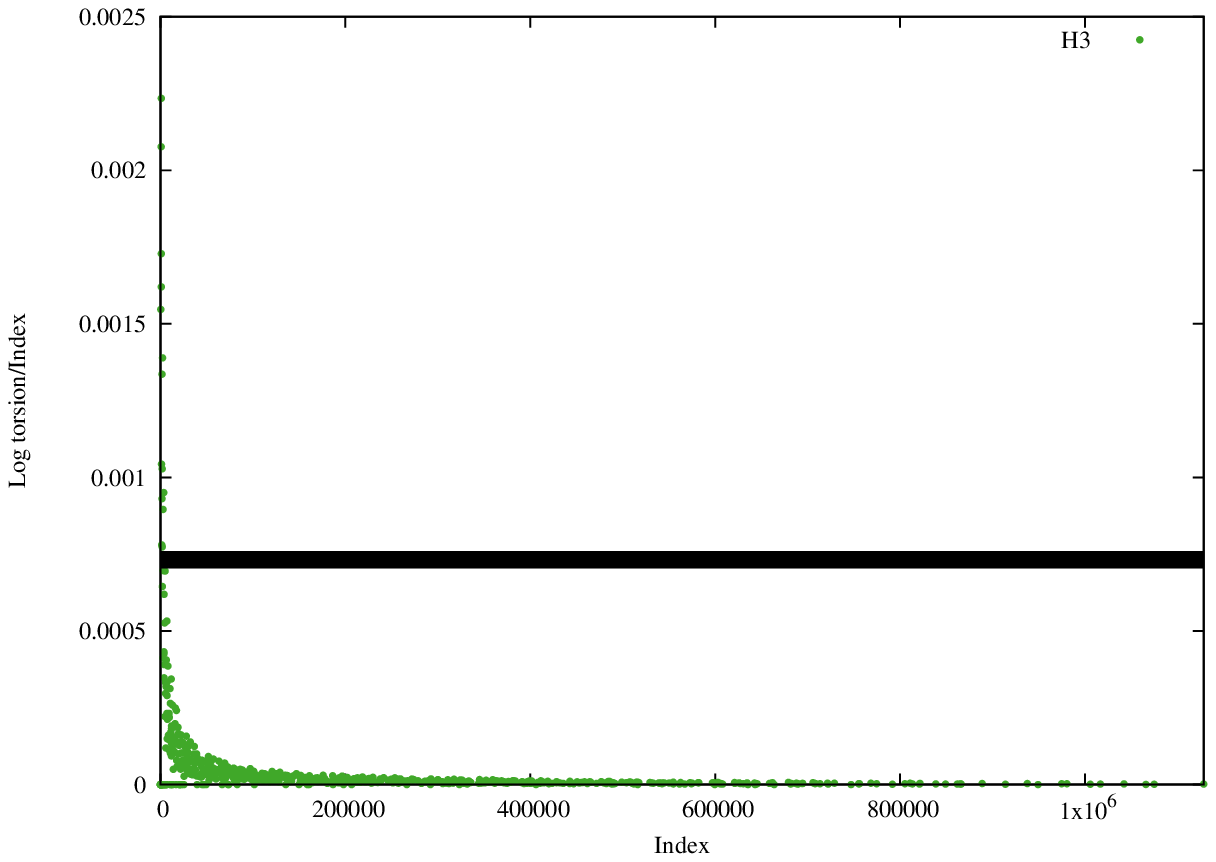}}
\quad\quad
\subfigure[$H_{4}$\label{gl3.4}]{\includegraphics[scale=0.5]{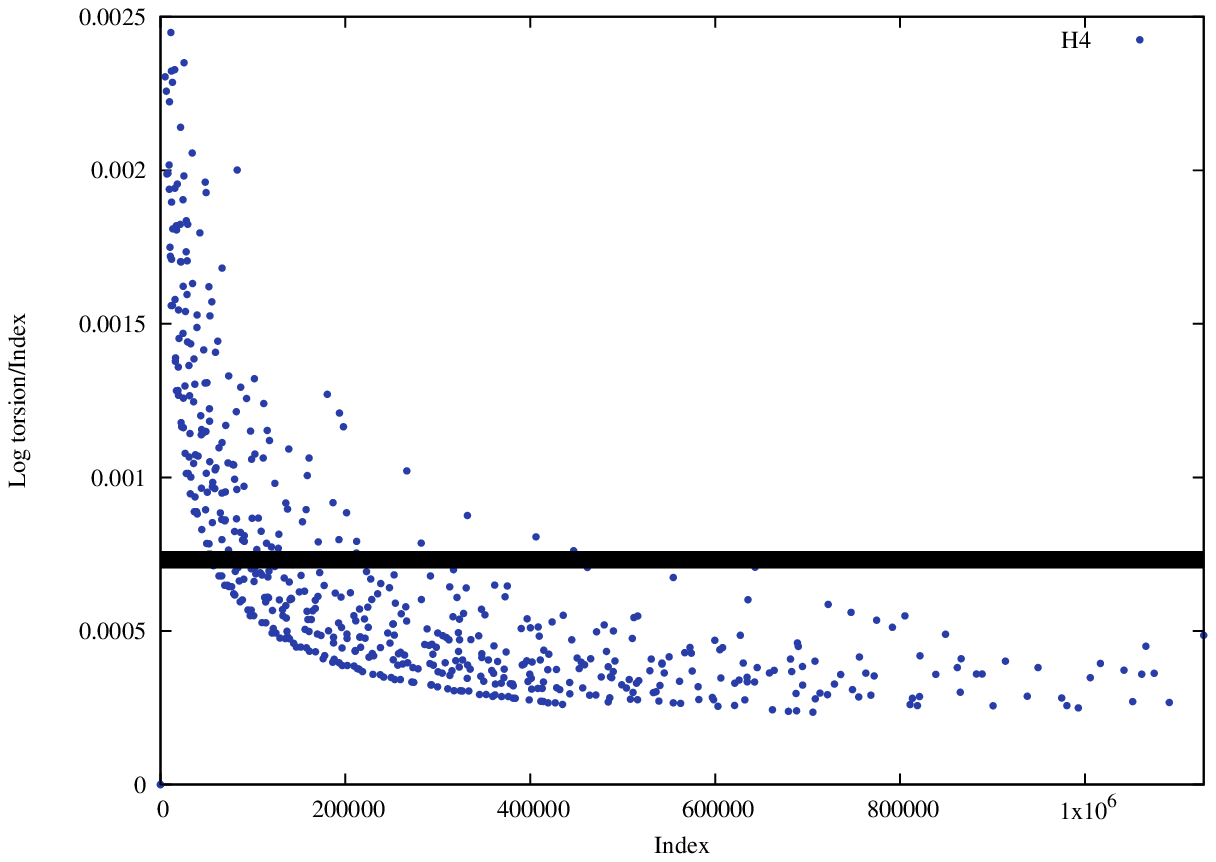}}
\end{center}
\caption{All the Voronoi homology groups for subgroups  of $\GL_{3}
(\ZZ)$, together with the predicted limiting constant (ordered by
index of the congruence subgroup).\label{fig:allgroupsgl3z}}
\end{figure}

\subsubsection{\texorpdfstring{$\GL_4 (\ZZ )$}{GL4(Z)}}
\begin{description}
\item[$H_3$] We computed $H_3$ for $118$ levels less than or equal to
  $119$.  Of these, there is nontrivial exotic torsion at $3$ levels.
  At level $114$, we have the largest torsion group in this range.
  The torsion size is $2^{12} \cdot 3^{7} \cdot \mathbf{11^{4}}$.  The
  largest exotic torsion and largest exotic 
  prime occurs at level $119$.  The full torsion is $2^{4} \cdot 3^{3}
  \cdot \mathbf{31^{4}}$.
\item[$H_4$]  We computed $H_4$ for $98$ levels less than $99$.  Of
  these, there is nontrivial exotic torsion at $2$ levels.  At level
  $49$, the torsion is $3 \cdot \mathbf{7^{2}}$.  At
  level $98$, the torsion is all exotic and has size is
  $\mathbf{7^5}$. 
\item[$H_5$]  We computed $H_5$ for $55$ levels less than $61$.  Of
  these, there is only nontrivial exotic torsion at level $49$.  At
  level $48$, we have the largest torsion group in this range.  The
  torsion size is $2^{418}\cdot 3$, all of which is nonexotic.  The
  largest exotic torsion and largest exotic prime occur at level
  $49$.  The full torsion is $2^{126} \cdot 3^{10} 
\cdot \mathbf{7}$.
\item[$H_6$] We computed $H_6$ for $71$ levels less than $74$.  There
  were no levels with nontrivial exotic torsion.  The largest torsion
  group occurs at level $50$.  The full torsion is $2^7$, all of
  which is nonexotic.
\item[$H_7$]  We computed $H_7$ for $78$ levels less than $80$.   There
  were no levels with nontrivial exotic torsion.  The largest torsion
  group occurs at level $80$.  The full torsion is $2^{165}$ all of
  which is nonexotic.
\end{description}

\begin{figure}[htb]
\begin{center}
\subfigure[$H_{3}$\label{gl4.3}]{\includegraphics[scale=0.5]{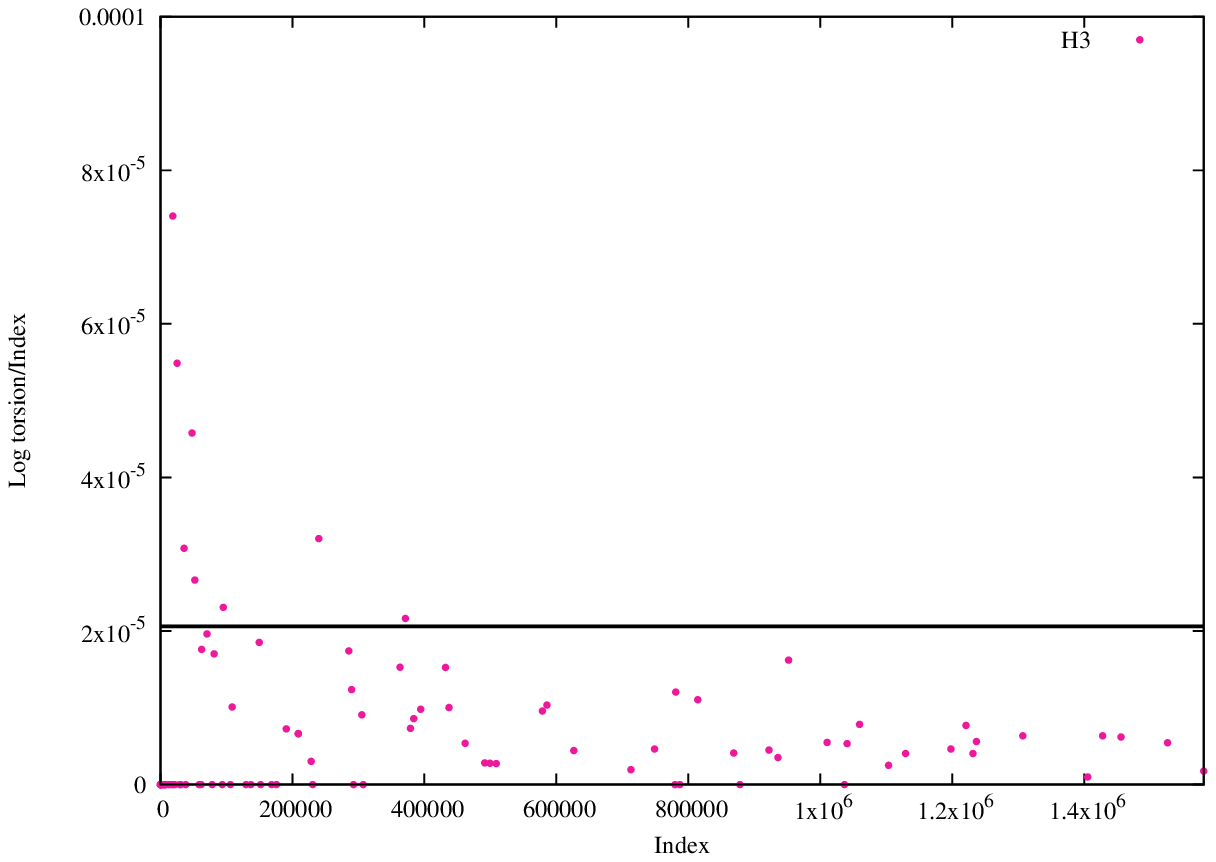}}
\quad\quad
\subfigure[$H_{4}$\label{gl4.4}]{\includegraphics[scale=0.5]{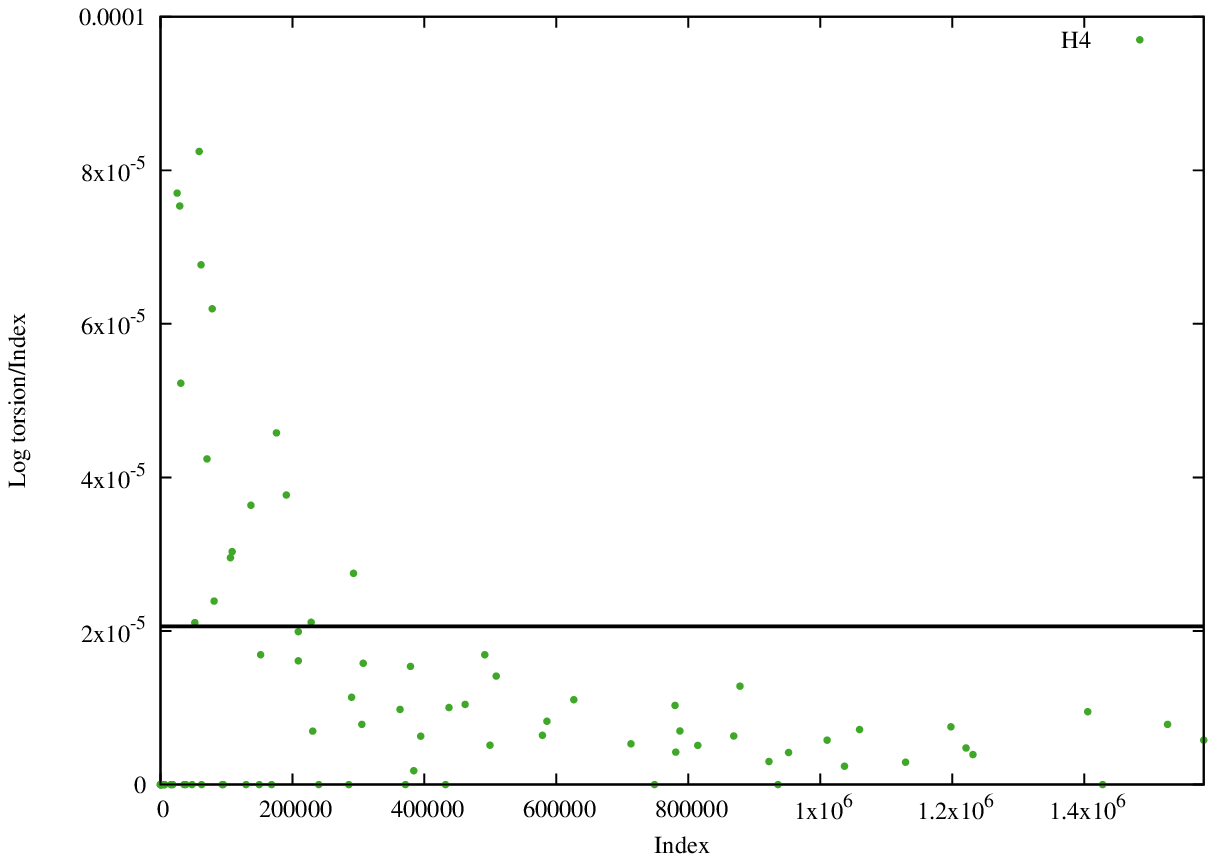}}
\quad\quad
\subfigure[$H_{5}$\label{gl4.5}]{\includegraphics[scale=0.5]{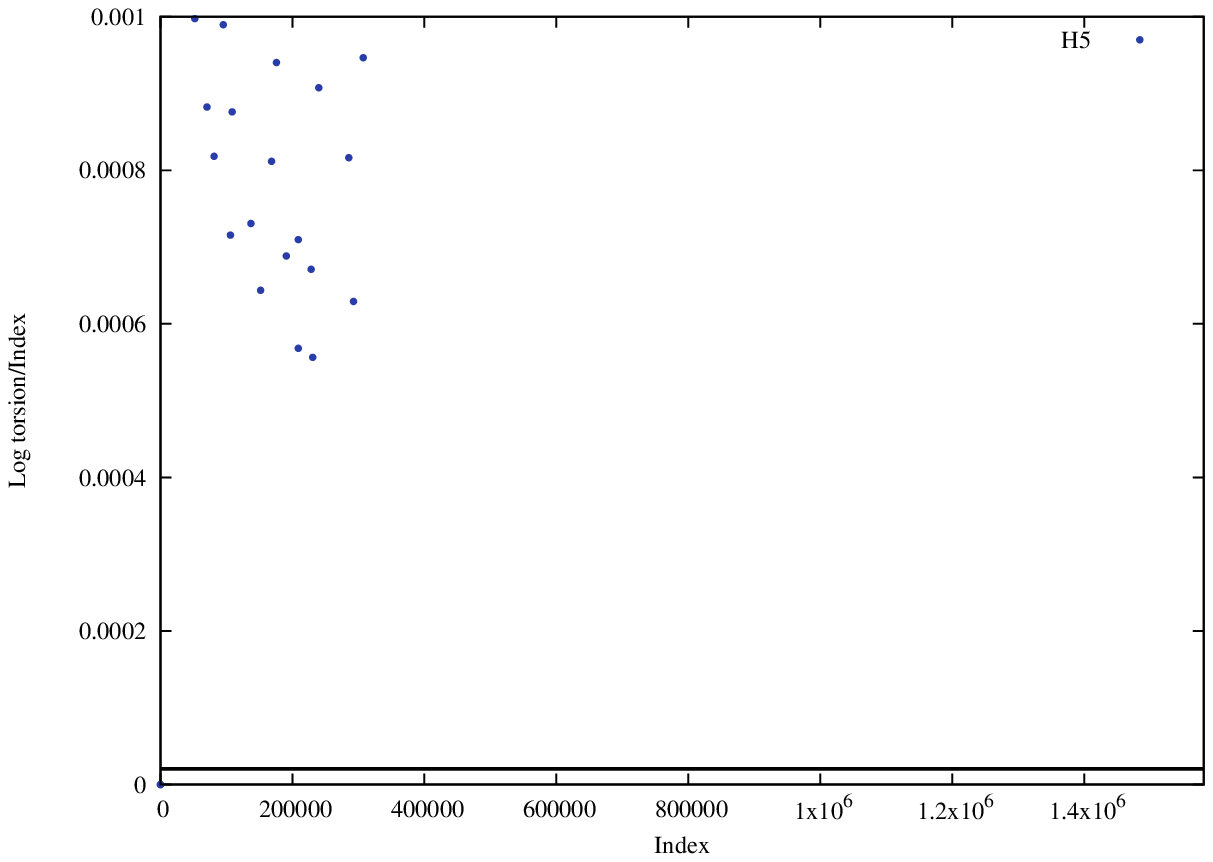}}
\quad\quad
\subfigure[$H_{6}$\label{gl4.6}]{\includegraphics[scale=0.5]{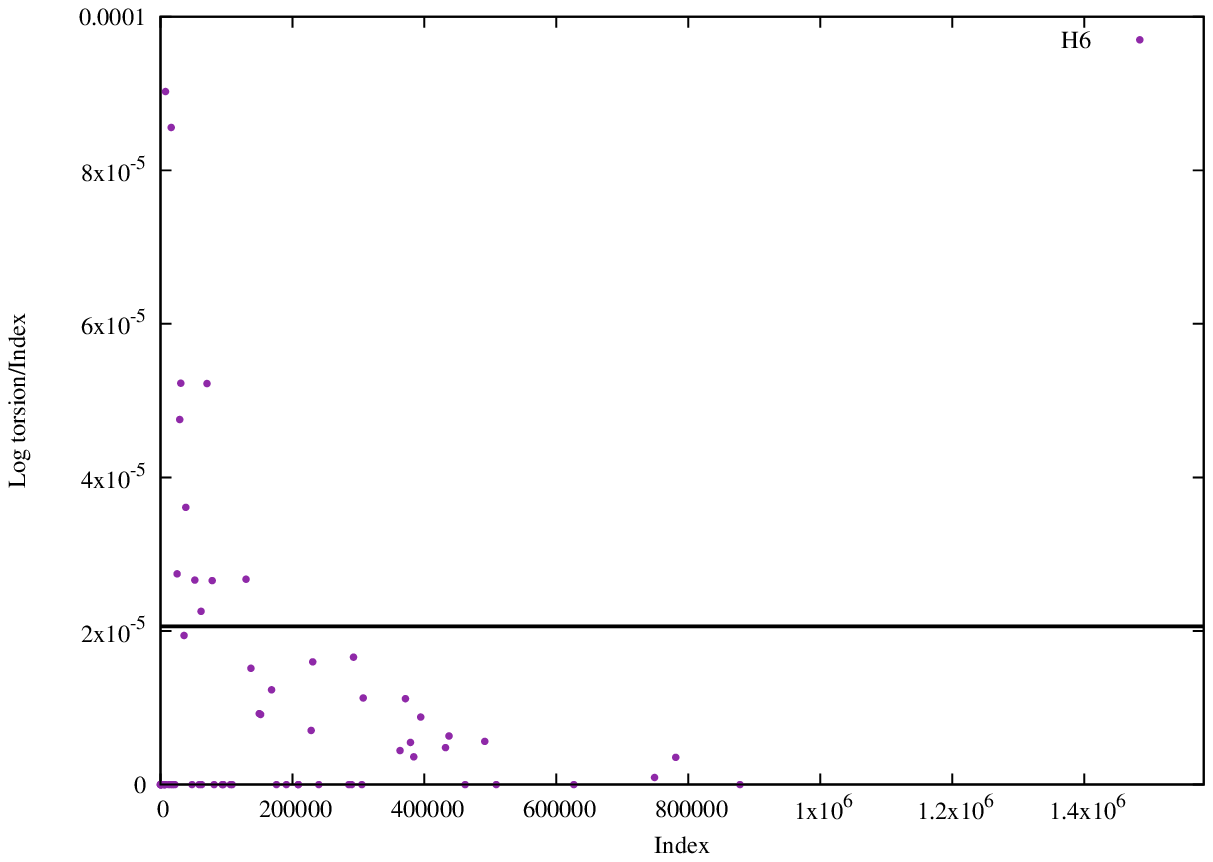}}
\quad\quad
\subfigure[$H_{7}$\label{gl4.7}]{\includegraphics[scale=0.5]{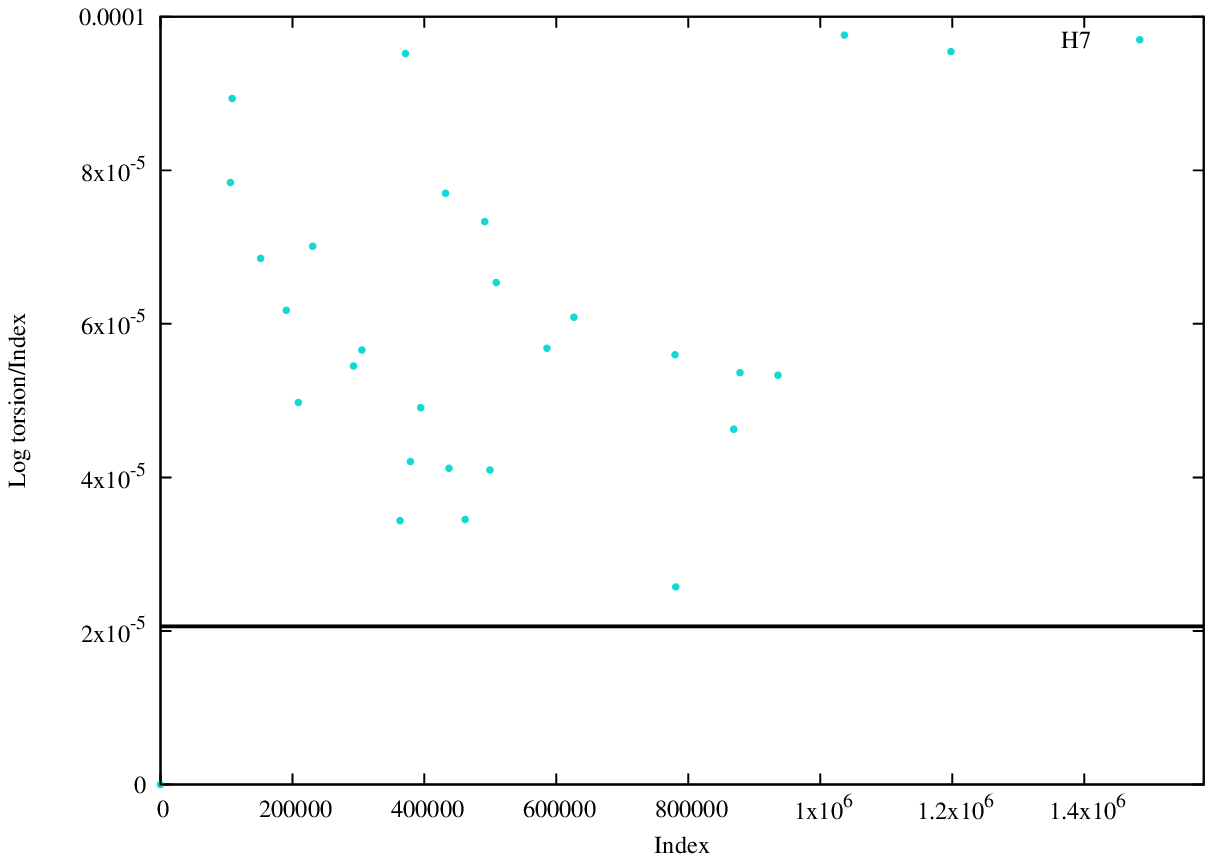}}
\end{center}
\caption{All the Voronoi homology groups for subgroups  of $\GL_{4}
(\ZZ)$, together with the predicted limiting constant (ordered by
index of the congruence subgroup).\label{fig:allgroupsgl4z}}
\end{figure}

\subsubsection{\texorpdfstring{$\GL_5(\ZZ )$}{GL5(Z)}}
\begin{description}
\item[$H_6$] We computed $H_6$ for 27 levels less than or equal to 29.  There were
no levels with nontrivial exotic torsion.  The largest torsion occurs
at level 24.  The torsion size is $2^{212}$.
\end{description}

\begin{figure}[htb]
\begin{center}
\includegraphics[scale=0.7]{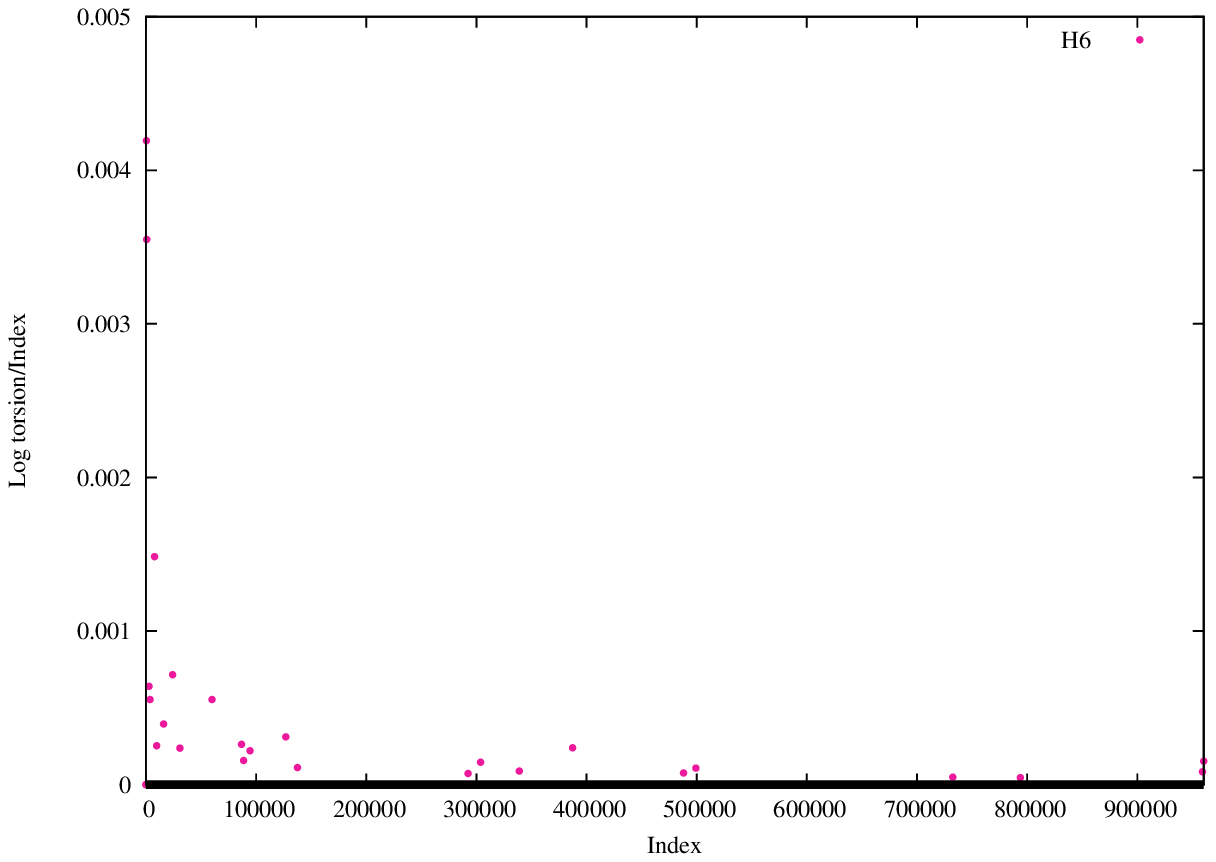}
\end{center}
\caption{The Voronoi homology group $H_{6}$ for subgroups of $\GL_{5}
(\ZZ )$ (ordered by index of the congruence subgroup)\label{gl5.6}}
\end{figure}

\subsection{\texorpdfstring{Cubic of discriminant $-23$}{Cubic
of discriminant -23}}

For the cubic field $F$ of discriminant $-23$, we were able to compute
homology groups up to the following level norms:

\medskip
\begin{center}
\begin{tabular}{|c||c|p{275pt}|}
\hline
Group&Deficiency $\delta$&Norm of level of congruence subgroup\\
\hline
$\GL_{2} (\OO_{F})$&1&$H_{1}$: $\Norm (\fn) \leq 5483$, $H_{2}$: $\Norm
(\fn) \leq 11575$, $H_{3}$: $\Norm (\fn) \leq 5480$, $H_{4}$: $\Norm
(\fn) \leq 5480$, $H_{5}$: $\Norm (\fn) \leq 5480$\\
\hline
\end{tabular}
\end{center}
\medskip
We pushed the computation further for $H_{2}$ since that is the top of
the cuspidal range from this group.  We highlight some of the results
of these computations.  As before, the torsion size is given in
factored form with exotic torsion in \textbf{bold}.
\begin{description}
\item[$H_1$] We computed $H_1$ for $2012$ levels of norm less than or
  equal to $5483$.  There were no levels with exotic torsion in this
  range.  The largest torsion group occurs at a level of norm $4481$.
  The torsion size is $2^3 \cdot 7$.
\item[$H_2$] We computed $H_2$ for $4240$ levels of norm less than or
  equal to $11575$.  Of these, there is nontrivial exotic torsion at
  $3374$ levels.  The largest torsion group occurs at a level of norm
  $10600$.  The torsion size is 
\[2^{3} \cdot 3^{15} \cdot \mathbf{5^{9}} \cdot 7^{2} \cdot \mathbf{11} \cdot \mathbf{103}.\]
  The largest exotic torsion occurs at norm level $8575$, where the
  full torsion has size
\[2^{4} \cdot \mathbf{5^{3}} \cdot \mathbf{7^{5}} \cdot \mathbf{13^{6}} \cdot \mathbf{139}.\]
The largest exotic prime occurs at norm level $11443$,
where the size of the torsion is exactly the exotic prime contribution 
$\mathbf{7870506841}$. 
\item[$H_3$] We computed $H_3$ for $2011$ levels of norm less than or
  equal to $5480$.  There were $1234$ levels with exotic torsion in this
  range.  The largest torsion group occurs at a level of norm $4928$.
  The torsion size is $2^{44} \cdot 3$, none of which is exotic.
  The largest exotic torsion occurs at norm level $4375$.  The full
  torsion has size $2^{6} \cdot \mathbf{5^{6}} \cdot \mathbf{19} \cdot \mathbf{31}$.
  The largest exotic prime occurs at
  norm level $4597$, where the torsion is $2 \cdot \mathbf{261529}$.
\item[$H_4$] We computed $H_4$ for $2010$ levels of norm less than or
  equal to $5480$.  There were no levels with exotic torsion in this
  range.  The largest torsion group occurs at a level of norm $2560$.
  The torsion size is $2^{16}$.
\item[$H_5$] We computed $H_5$ for $2010$ levels of norm less than or
  equal to $5480$.  There were no levels with exotic torsion in this
  range.  The largest torsion group occurs at a level of norm $3325$.
  The torsion size is $3^4$.
\item[$H_6$] We computed $H_6$ for $2010$ levels of norm less than or
  equal to $5480$.  There were no levels with nontrivial torsion in this
  range.  
\end{description}

\begin{figure}[htb]
\begin{center}
\subfigure[$H_{1}$\label{23.1}]{\includegraphics[scale=0.5]{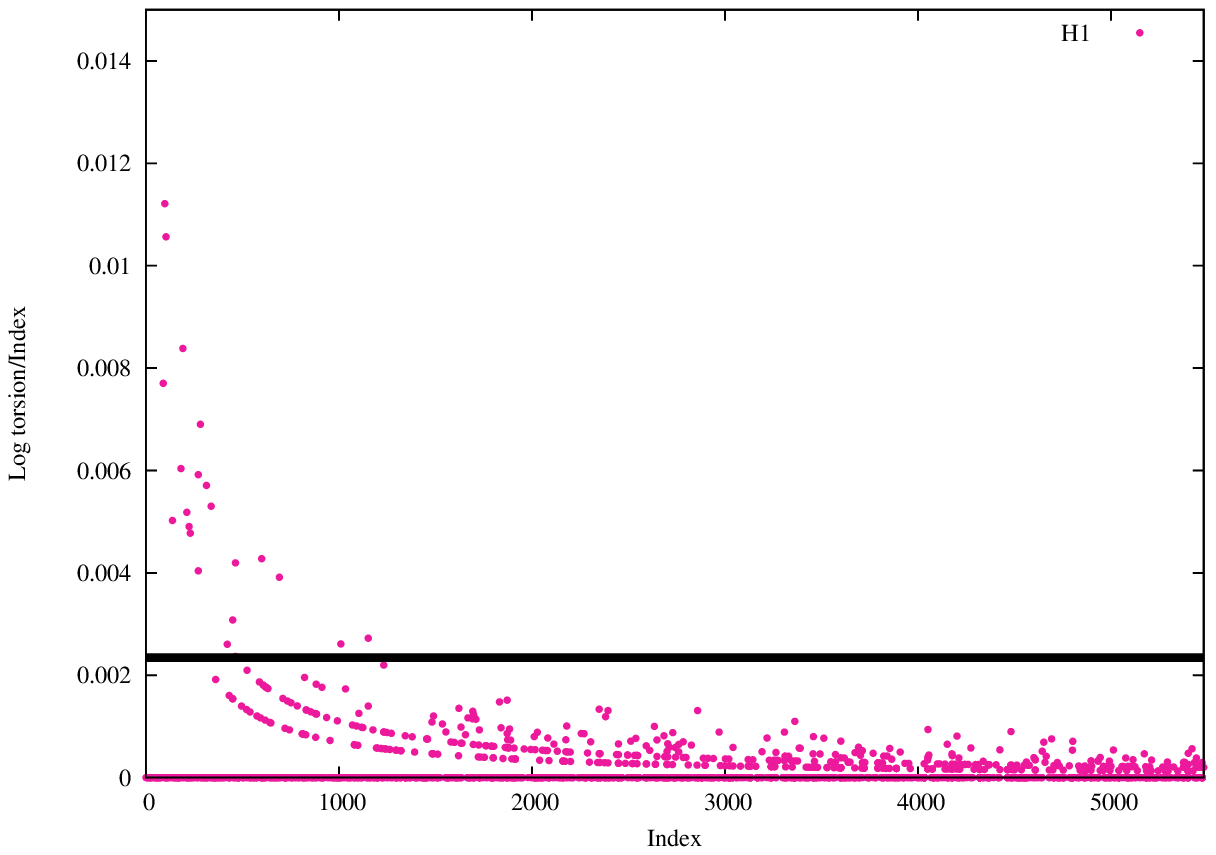}}
\quad\quad
\subfigure[$H_{2}$\label{23.2}]{\includegraphics[scale=0.5]{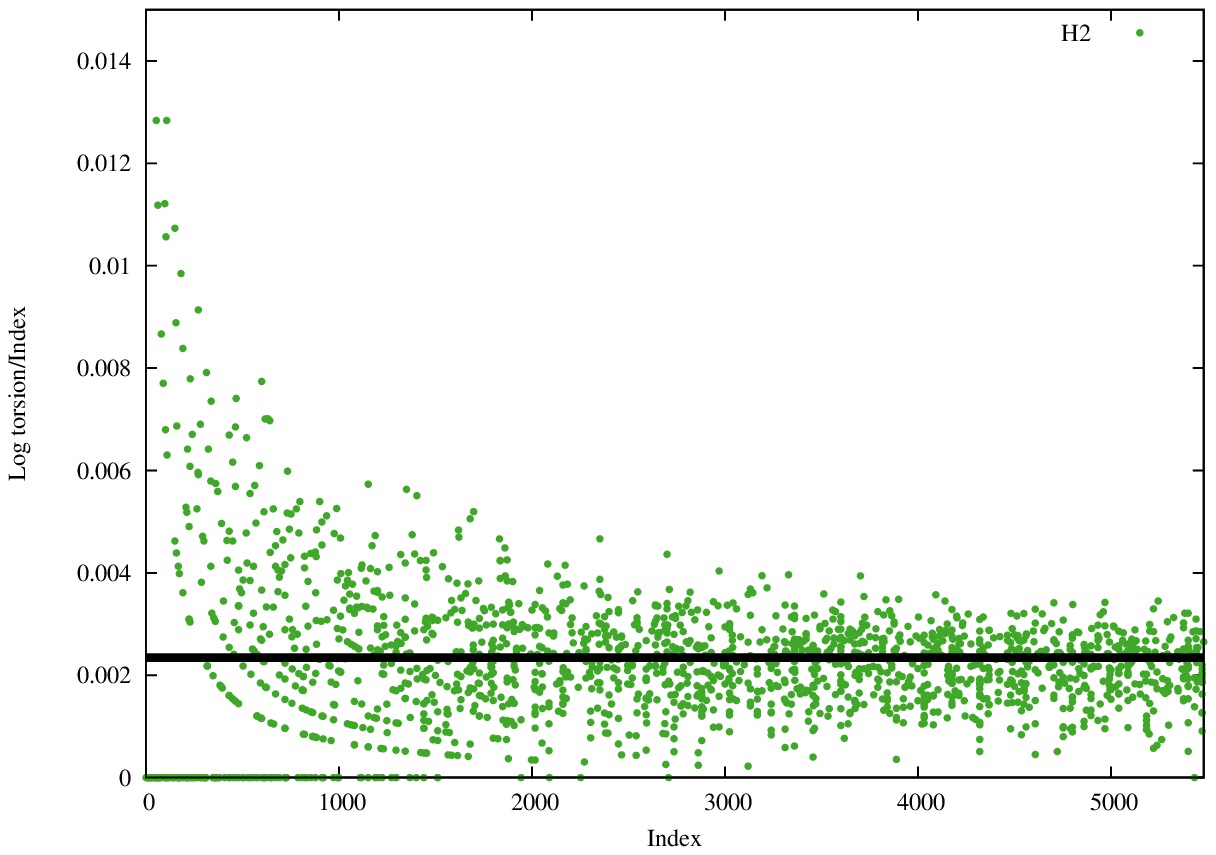}}
\quad\quad
\subfigure[$H_{3}$\label{23.3}]{\includegraphics[scale=0.5]{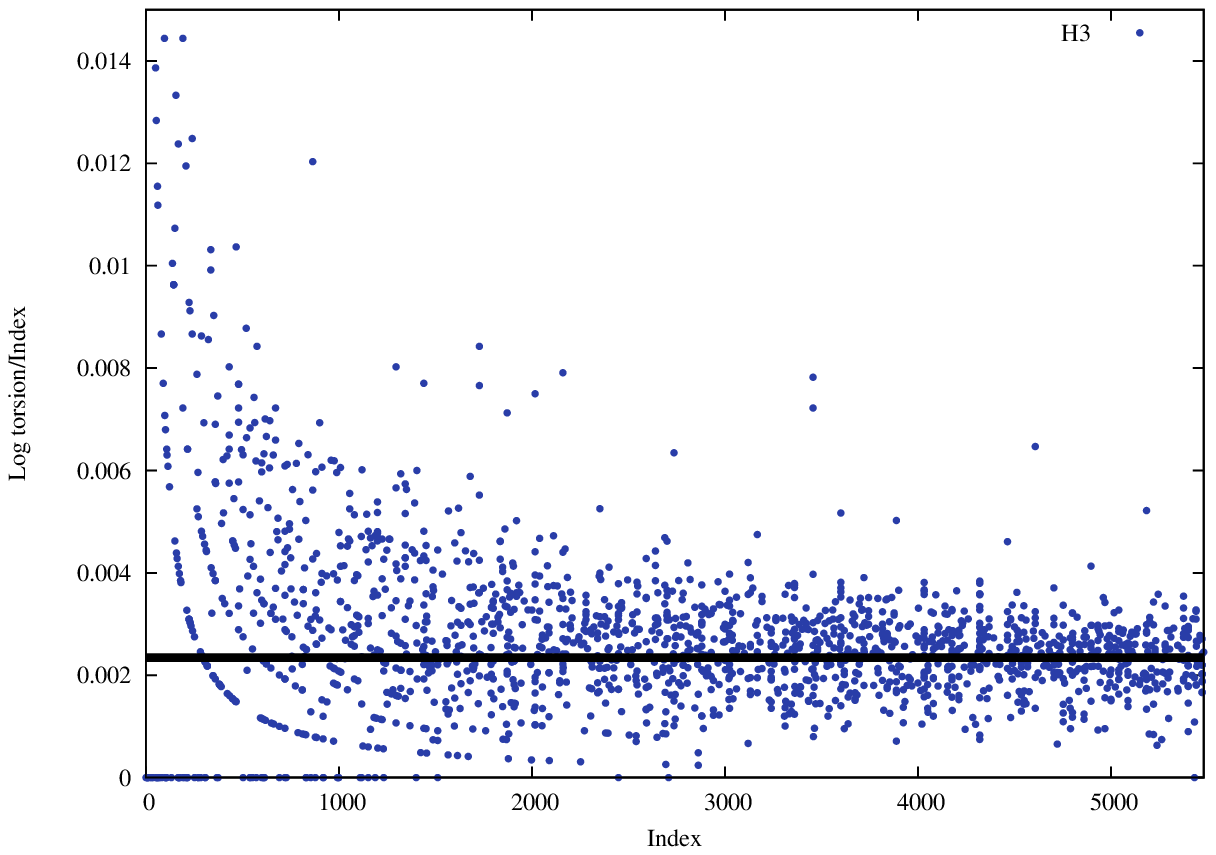}}
\quad\quad
\subfigure[$H_{4}$\label{23.4}]{\includegraphics[scale=0.5]{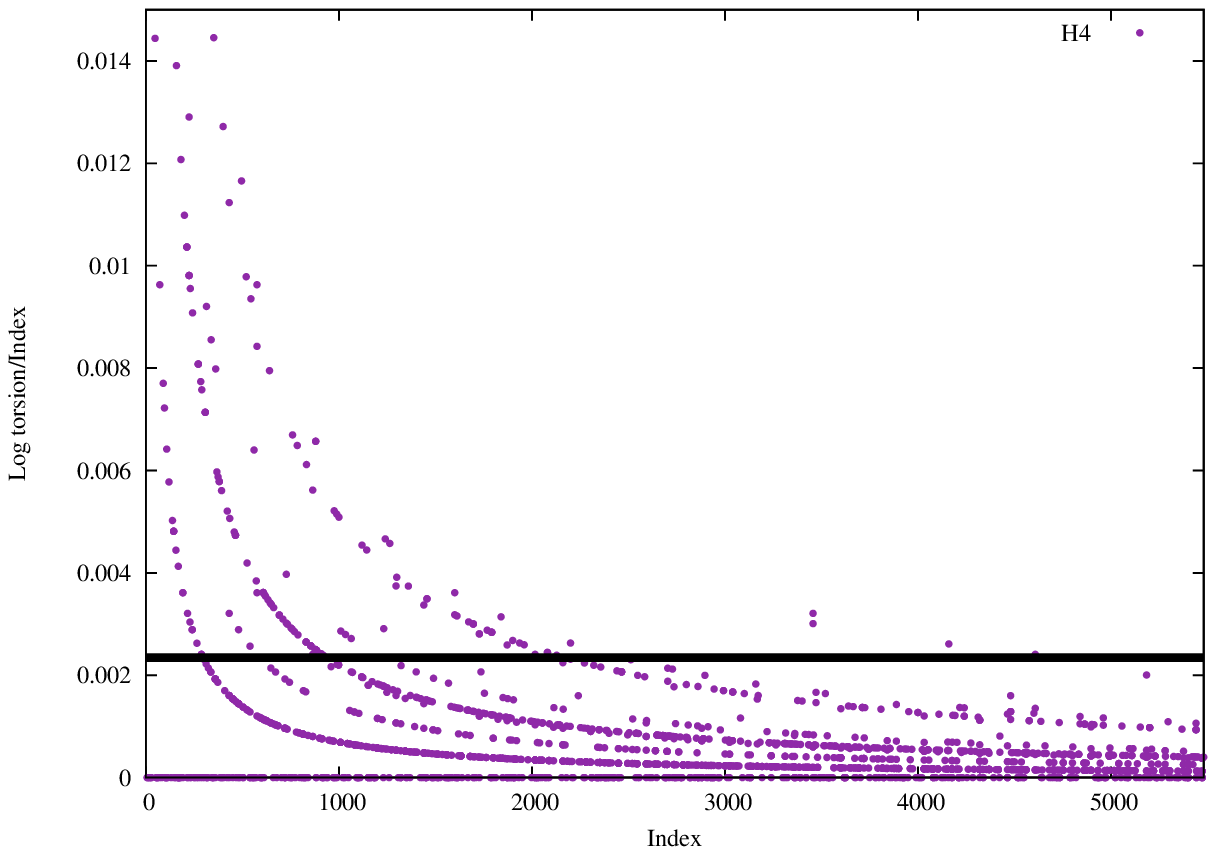}}
\quad\quad
\subfigure[$H_{5}$\label{23.5}]{\includegraphics[scale=0.5]{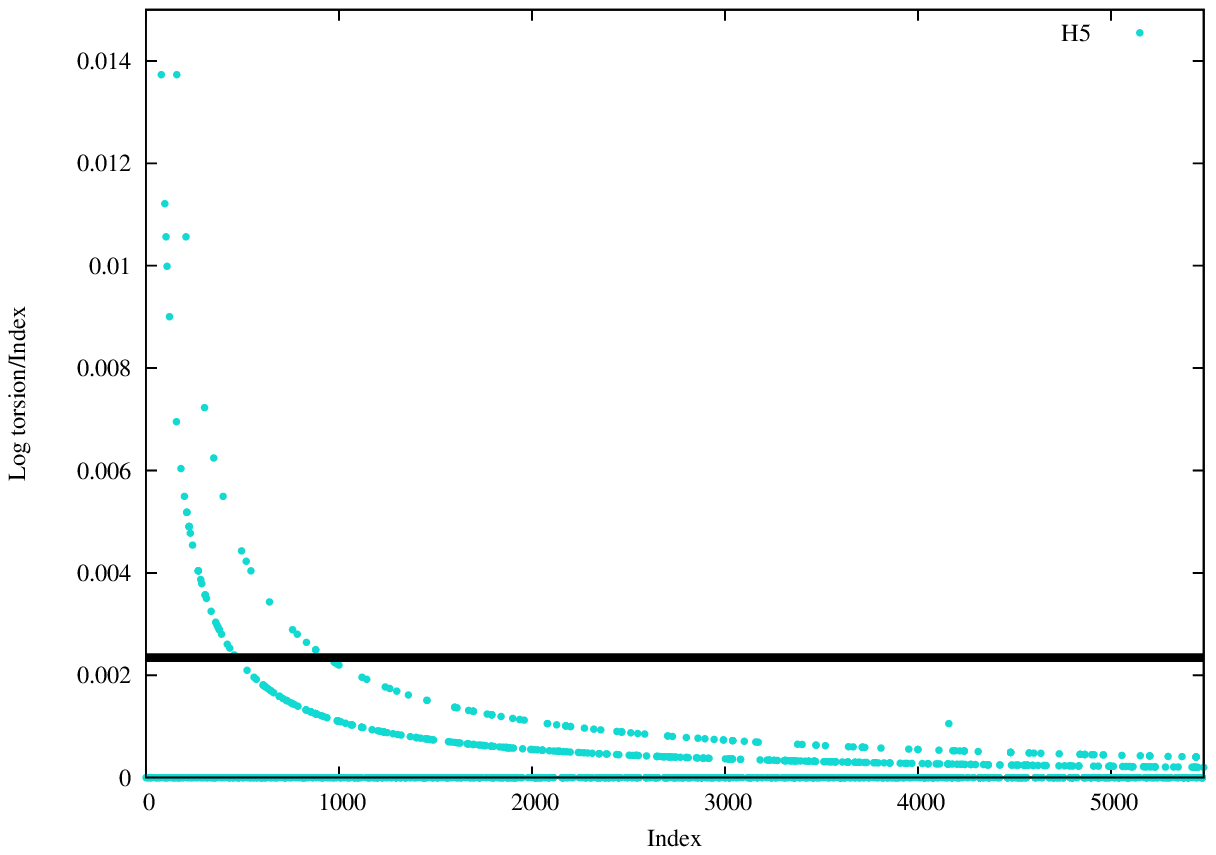}}
\end{center}
\caption{All the Voronoi homology groups for subgroups of $\GL_{2} (\OO_{F})$ for the cubic
field of discriminant $-23$, together with the predicted limiting
constant (ordered by index of the congruence subgroup).\label{fig:allgroupsneg23}}
\end{figure}

\begin{figure}[htb]
\begin{center}
\includegraphics[scale=0.75]{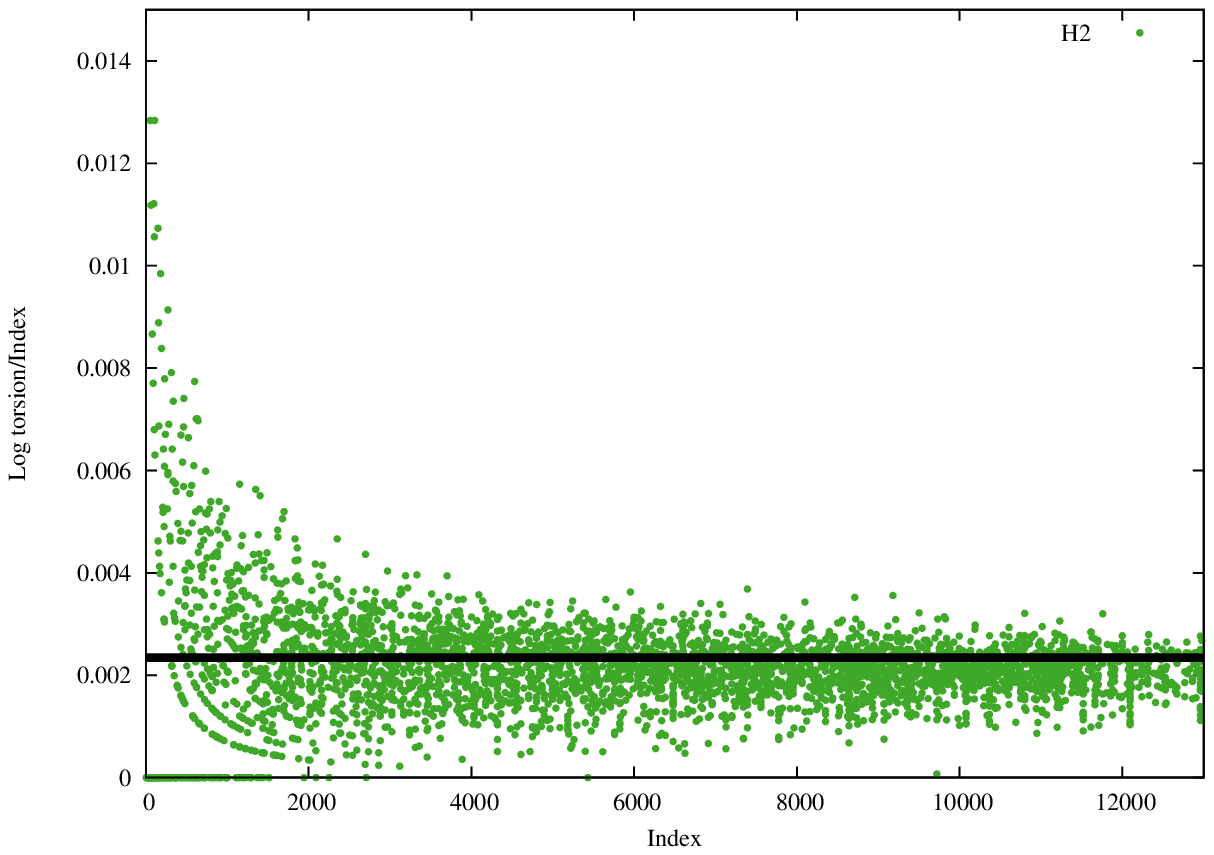}
\end{center}
\caption{$H_{2}$ of subgroups of $\GL_{2} (\OO_{F})$ for the cubic
field of discriminant $-23$, together with the predicted limiting
constant (ordered by index of the congruence subgroup).  This includes many more
levels than Figure \ref{23.2}\label{fig:bighomology}.}
\end{figure}

\subsection{Field of fifth roots of unity}
For this field, we only concentrated on $H_{2}$, which is the top of
the cuspidal range.  The
deficiency of this group is $\delta =2$.
We computed $H_2$ for levels of norm up to $31805$.  This includes
$2741$ levels.  Of these, there is nontrivial exotic torsion at $239$
levels.  We highlight some of the results
of these computations.  As before, the torsion size is given in
factored form with exotic torsion in \textbf{bold}.

The largest torsion group occurs at a level of norm 
$15625$ and is nonexotic. The order is $5^{15}$.  The largest
exotic torsion occurs at norm level $24025$, where the full torsion is
$5^{3} \cdot \mathbf{7} \cdot \mathbf{11} \cdot \mathbf{31^{2}}
\cdot \mathbf{61}$.
The largest exotic prime occurs at norm level $17161$, where the size
of the torsion is exactly the exotic prime contribution $\mathbf{2081}$.

\begin{figure}[htb]
\begin{center}
\includegraphics[scale=0.75]{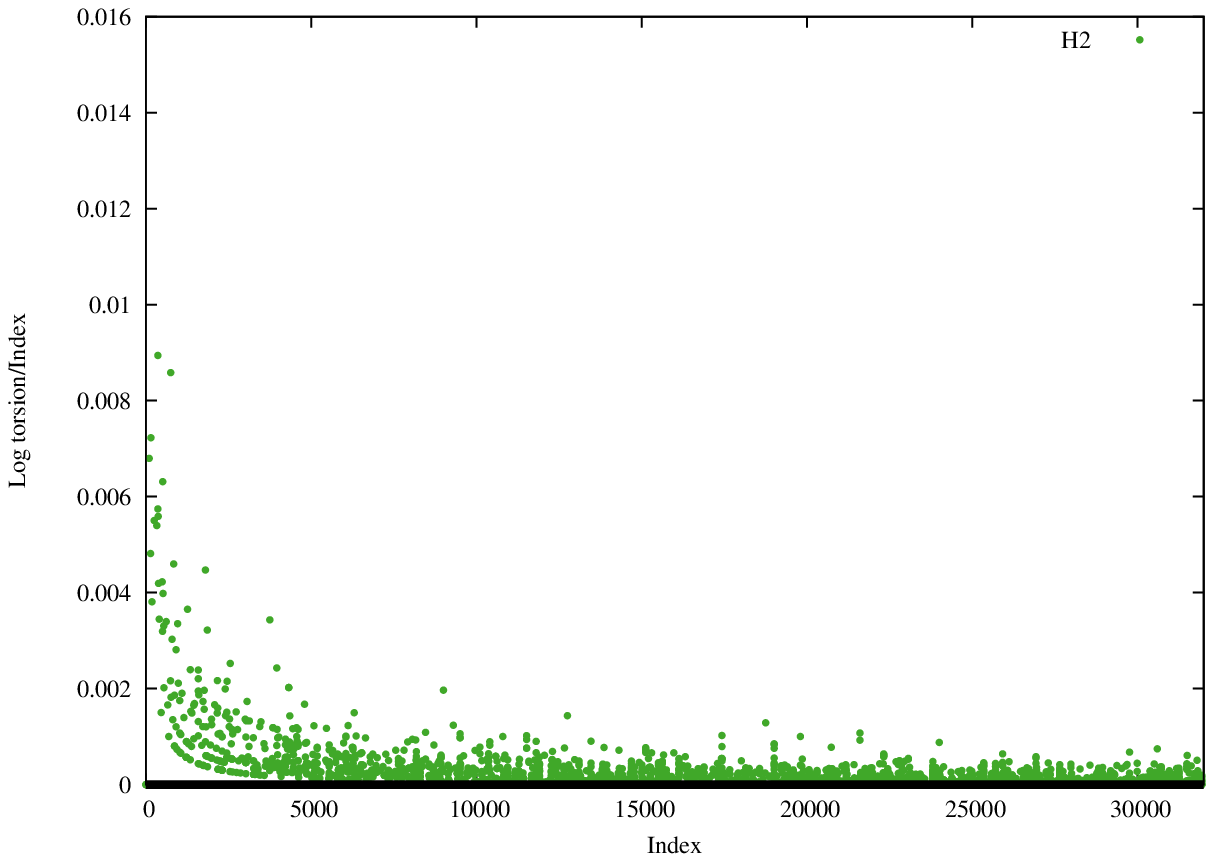}
\end{center}
\caption{The Voronoi homology group $H_{2}$ for subgroups of $\GL_{2}
(\OO_{E})$, $E$ the field $\QQ (\zeta_{5})$ (ordered by index of the
congruence subgroup).}
\end{figure}

\subsection{Imaginary quadratic fields}\label{ss:imag}
In a recent work, \sengun \cite{haluk.bianchi} computed $H^{2}$ of
congruence subgroups of $\PSL_{2} (\OO_{L})$ for the euclidean
imaginary quadratic fields $\QQ (\sqrt{-d})$, $d=1, 2, 3, 7, 11$.
These groups all have deficiency $\delta =1$. He computed cohomology
for prime ideals of norm $\leq 5000$, and also computed the cohomology
for a variety of nontrivial coefficient systems.  Our results when
spot checked agree completely with his, although we have not
systematically compared data.  We give plots in
Figure~\ref{fig:allgroupsbianchi} for the other fields we considered,
namely $\QQ (\sqrt{-d})$, $d=5,6,10,13,14,15$.  We remark that these
fields have class number $2,2,2,2,4,2$.  As the plots show, the
torsion is enormous in $H_1$.  For example, $8303$ levels of norm less
than or equal to $10103$ were considered for $\QQ(\sqrt{-15})$.  In
this range, the largest torsion group is on the order of $10^{1126}$.
The largest exotic contribution is on the order of $10^{924}$.
Because the torsion is so large, we are unable to factor the sizes to
report on the largest exotic primes in the range of our computation.

\begin{figure}[htb]
\begin{center}
\subfigure[$\QQ (\sqrt{-5})$\label{bianchi5}]{\includegraphics[scale=0.5]{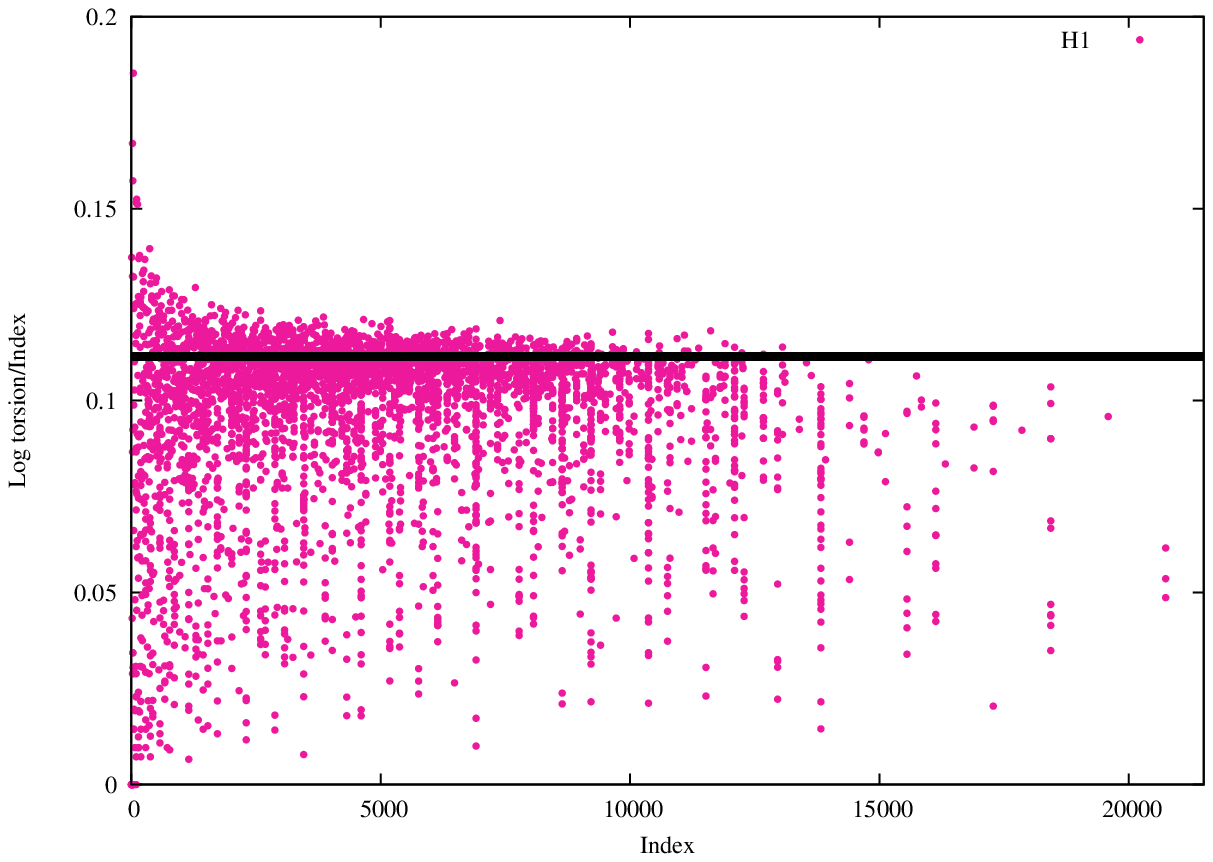}}
\quad\quad
\subfigure[$\QQ (\sqrt{-6})$\label{bianchi6}]{\includegraphics[scale=0.5]{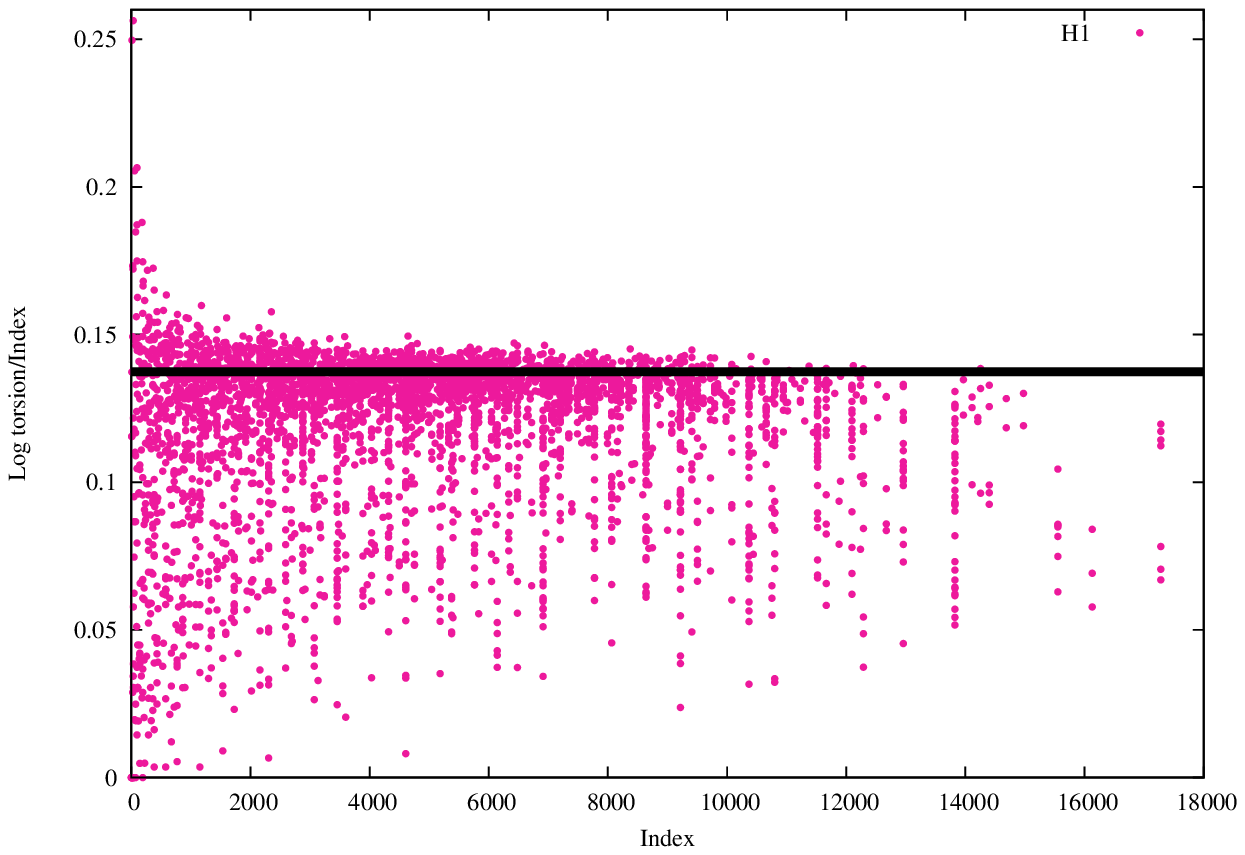}}
\quad\quad
\subfigure[$\QQ (\sqrt{-10})$\label{bianchi10}]{\includegraphics[scale=0.5]{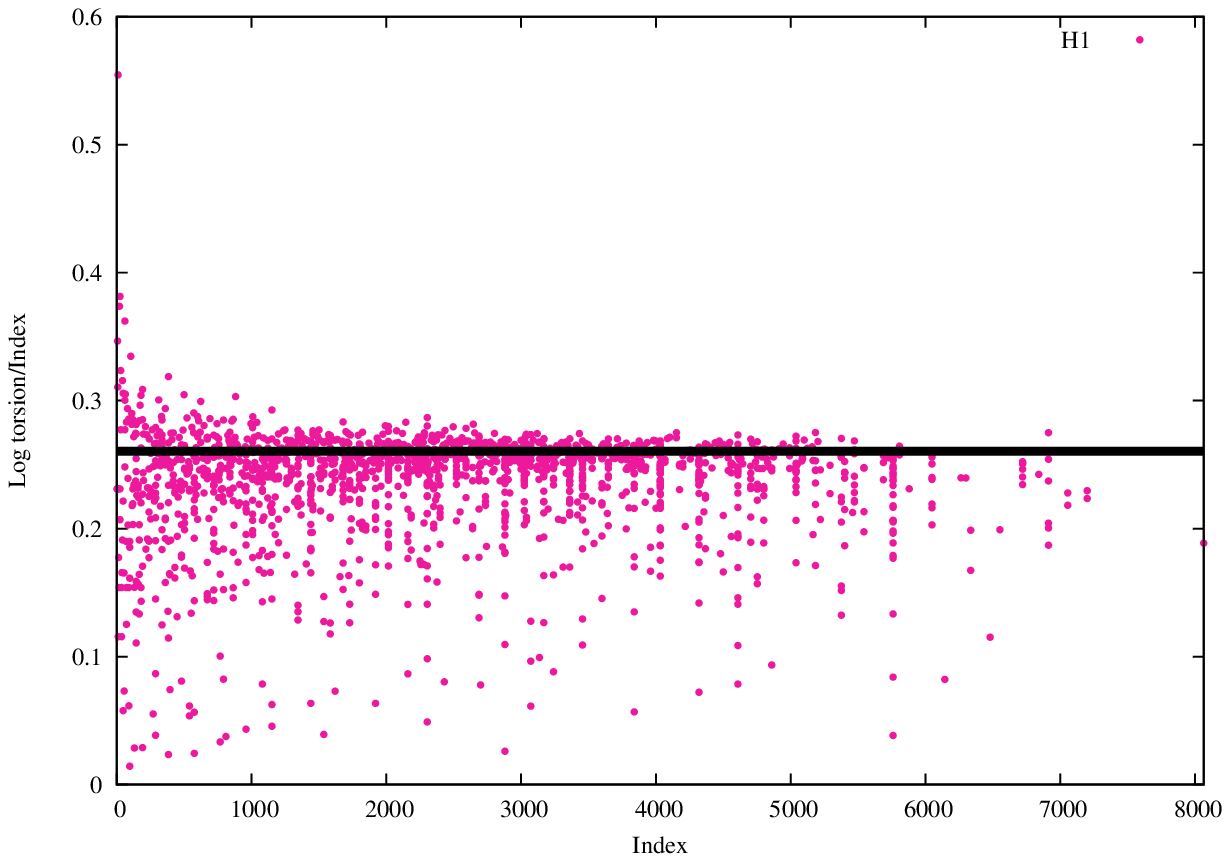}}
\quad\quad
\subfigure[$\QQ (\sqrt{-13})$\label{bianchi13}]{\includegraphics[scale=0.5]{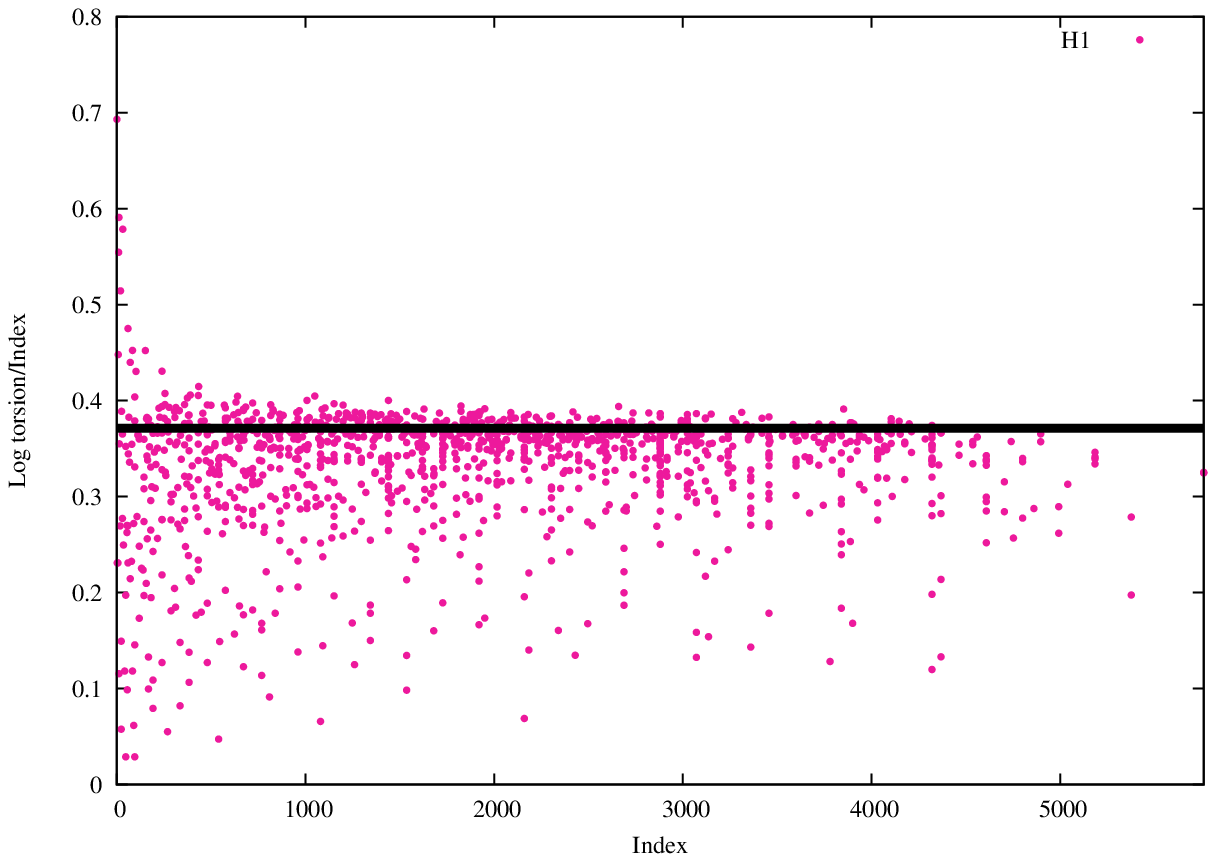}}
\quad\quad
\subfigure[$\QQ (\sqrt{-14})$\label{bianchi14}]{\includegraphics[scale=0.5]{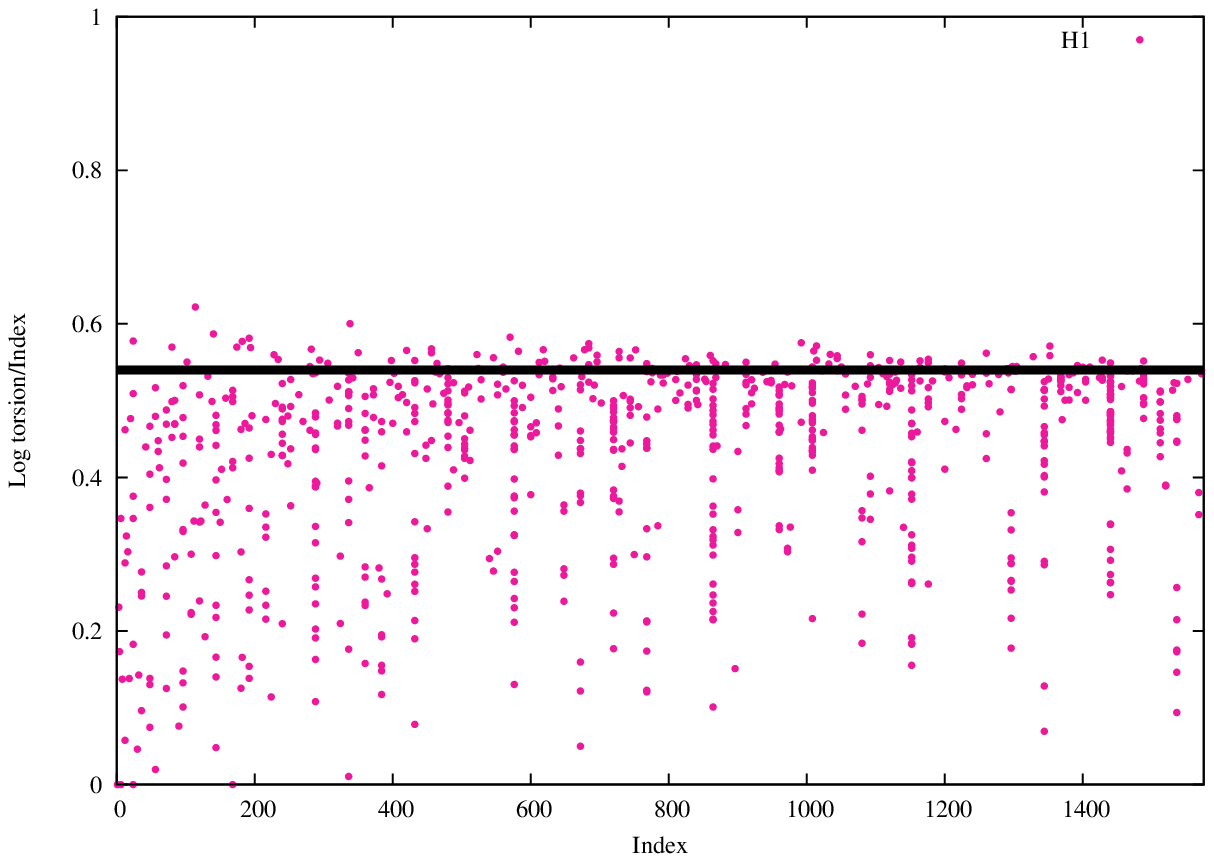}}
\quad\quad
\subfigure[$\QQ (\sqrt{-15})$\label{bianchi15}]{\includegraphics[scale=0.5]{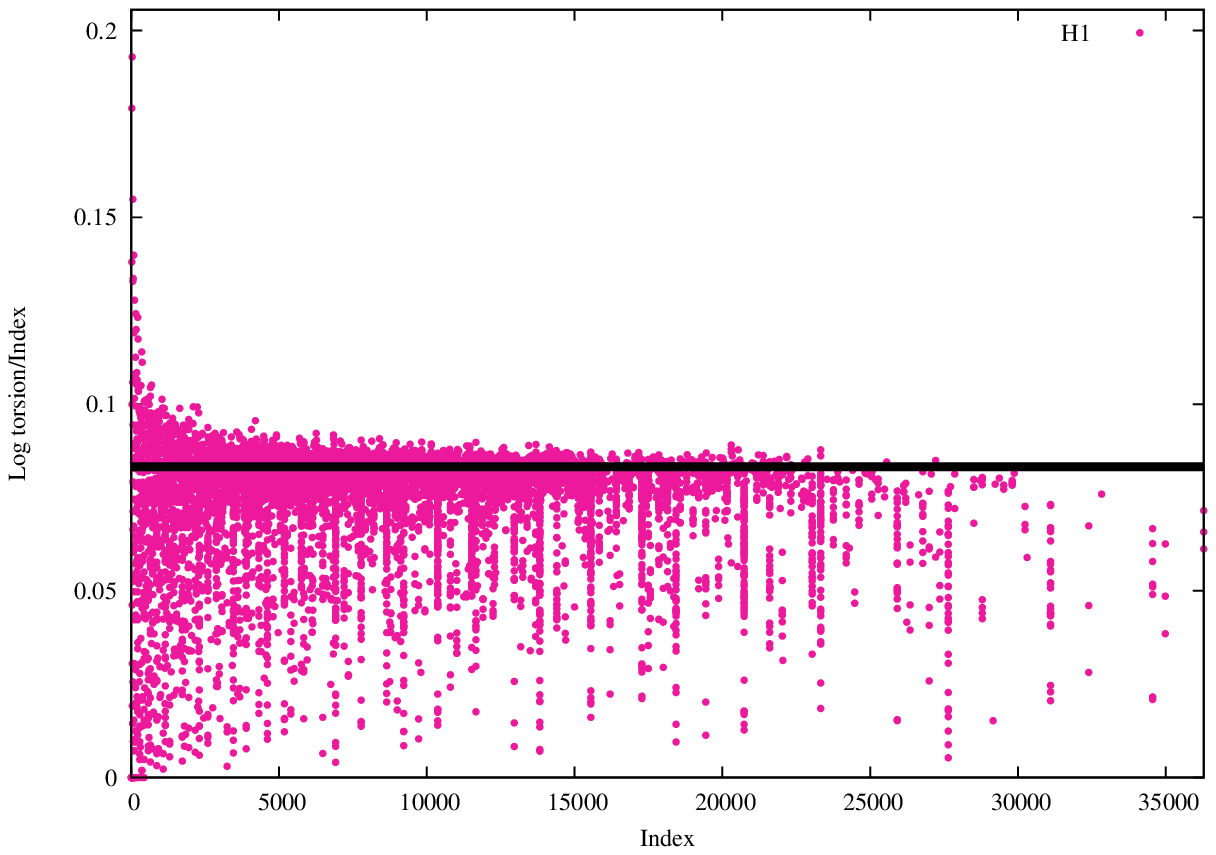}}
\end{center}
\caption{The Voronoi homology group $H_{1}$ for subgroups of $\GL_{2}
(\OO_{L})$ for various imaginary quadratic fields $L$, together with
the predicted limiting constant (ordered by index of the congruence subgroup).  These fields have class number
$2,2,2,2,4,2$ respectively.\label{fig:allgroupsbianchi}}
\end{figure}
\section{Cuspidal range}\label{s:cuspidalrange}

\subsection{}
Recall that the \emph{cuspidal range} for an arithmetic group is the
set of cohomological degrees for which cuspidal automorphic forms
can contribute to the corresponding cohomology.  For a discussion,
see \cite[\S 2]{schwermer.survey}.  It is known that if a cuspidal
automorphic form contributes to one degree in the cuspidal range, then
it does to all, and the relevant Hecke eigensystems are determined by
the contribution to the top (cf.~\cite[Theorem
2.6]{schwermer.survey}).  Thus it is of interest to see if similar
phenomena occur for torsion in cohomology: if the deficiency $\delta$
is 1 and we see exotic torsion in the cohomology, do we see it across
the cuspidal range, or is there an apparent difference among degrees
in the range?  What if the deficiency is greater than $1$?

\subsection{}
We observed in our computations that there is a distinction between
the torsion phenomena across the cohomological degrees.  In
particular, we always observed an abundance of exotic torsion \emph{in
the top degree only}.  This is consistent with earlier work on
imaginary quadratic fields \cite{priplata,haluk.bianchi}: the cuspidal
range in that case is $H^{1}$---$H^{2}$, and only $H^{2}$ sees the big
torsion.  For an example from our data, consider the plots for $\GL_{3}/\QQ$
(Figures~\ref{gl3.2}---\ref{gl3.4}).  The cuspidal range corresponds
to $H_{2}$ (Figure \ref{gl3.2}) and $H_{3}$ (Figure \ref{gl3.3}).
These figures clearly show lots of torsion in $H_{2}$, little in
$H_{3}$, and the top of the cuspidal range corresponds to $H_{2}$.

Thus our data suggests that, when Conjecture \ref{conj:bv} is
generalized to an arithmetic group $\Gamma$ in a \emph{semisimple}
$\QQ$-group $\bG$ with nontrivial $\QQ$-rank, the limit of the ratios
\[
\frac{\log |H^{i}(\Gamma_{k} ;
\sL)_{\tors}|}{[\Gamma : \Gamma_{k}]}
\]
should vanish unless $\delta =1$ and $i$ is the \emph{top degree of the
cuspidal range}, and the nonzero limit should be the quantity
$c_{G,\sL}$.  In the case of a \emph{reductive} group $\bG$ with
positive-dimensional split component $A_{G}$, one should see more
cohomology groups with a nonzero limit; see the discussion below in \S
\ref{s:chi} for an explanation.

\subsection{}
Given that the abundance of exotic torsion occurs in the top degree of
the cuspidal range, what can one say about other degrees?  In lower
degrees, we observed that there can still be exotic torsion, as
defined here, but we did not see the exponential growth or the
appearance of gigantic torsion primes.  

Our data behave differently for different groups.  For example, our
data shows that for $\GL_{3}/\QQ$, \emph{all the exotic primes in
$H_{3}$ actually divide the level}.  We don't have an explanation for
this.  On the other hand, for $\GL_{2}/F$ where $F$ is the $-23$
cubic, we do see exotic primes in $H_{3}$ as well as $H_{2}$, but they
coincide.  More precisely, the torsion subgroups of $H_{3}$ and
$H_{2}$ are in general different, but the portion of them
corresponding to exotic primes has the same order in all cases.  This
should be another manifestation of the effect of the flat factor on
the cohomology, as described in \S\ref{ss:kunneth}.  In addition we
note that $H_{4}$ for the $-23$ cubic, which is the bottom of the
cuspidal range, has no exotic torsion at all.

\subsection{}
Finally, we consider the case of deficiency $\delta > 1$.  For the
group $\GL_{2} (\OO_{E})$, $E = \QQ (\zeta_{5})$, we did observe
exotic torsion in $H_{2}$, which is the top of the cuspidal range.
These primes start out small, and don't appear to grow as fast as in
the deficiency $1$ case.  However they do grow (although the data,
which is omitted here, is somewhat equivocal).  It remains an
interesting problem to formulate precisely how rapidly the exotic
torsion grows for $\delta >1$.

\section{Alternating sum (``Euler characteristic'') plots}\label{s:chi}

\subsection{}
Recall that Bergeron--Venkatesh were able to prove an ``Euler
characteristic'' version of their conjecture \eqref{eq:bvthm}.  In
particular, although they were unable to show that their Conjecture
\ref{conj:bv} held, they could show that the alternating sum of the
ratios $\log |H_{\tors }^{i} (\Gamma_{k})|/[\Gamma : \Gamma_{k}]$
tended to the predicted constants.  Thus we applied the alternating
sum to our data to see if this provided better apparent matching with
the predicted limiting constant.  All groups we checked have
deficiency $\delta =1$.

\subsection{}

For $\GL_{3}/\QQ$, one notices that the ratios in Figure \ref{gl3.2}
for $H_{2}$ appear to creep up to the limiting constant, whereas those
in Figure \ref{gl3.4} for $H_{4}$ simultaneously fall down to the
constant.  The group $H_{4}$ consists of large amounts of $2$- and
$3$-torsion, and is outside the cuspidal range, whereas the former is
replete with exotic torsion and is at the top of the cuspidal range.
The group $H_{3}$, on the other hand, quickly contributes nothing to
the computation (Figure \ref{gl3.3}).  Thus together $H_{2}$ and
$H_{4}$ combine to give a large contribution to the predicted limit,
as seen in Figure \ref{fig:chigl3z}. It is possible that for this
group the limit of the Euler characteristic might be larger than the
the B-V limit.

\begin{figure}[htb]
\begin{center}
\includegraphics[scale=0.75]{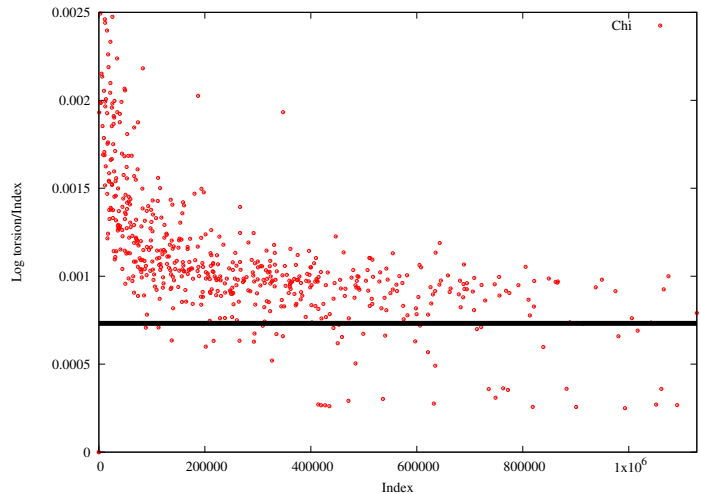}
\end{center}
\caption{The ``Euler characteristic'' for subgroups of $\GL_{3} (\ZZ )$\label{fig:chigl3z}}
\end{figure}

\subsection{}

For $\GL_{4}/\QQ$ we haven't enough data to show a clear picture of
the trending of the Euler characteristic.  We include the plot in
Figure \ref{fig:chigl4z} for completeness.

\begin{figure}[htb]
\begin{center}
\includegraphics[scale=0.75]{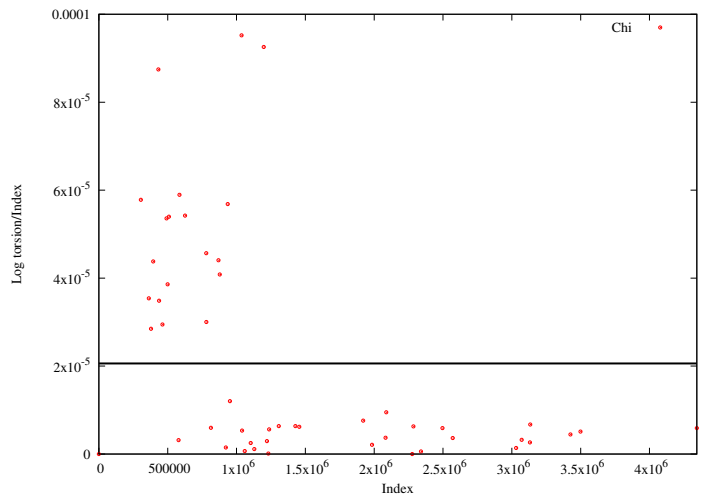}
\end{center}
\caption{The ``Euler characteristic'' for subgroups of $\GL_{4} (\ZZ )$\label{fig:chigl4z}}
\end{figure}

\subsection{}\label{ss:kunneth}

Finally, we consider the group $\GL_{2}/F$, $F$ the $-23$ cubic.  In
this case we have an additional flat factor in the symmetric space,
and the Euler characteristic computation causes an amusing
cancellation: the exotic torsion that should properly be at the top of
the cuspidal range in $H_{2}$ also appears in $H_{3}$, the group one
degree below.  As remarked at the end of \S\ref{s:cuspidalrange}, the
order of the exotic torsion summand is the same for $H_{2}$ and
$H_{3}$.  This is caused by the K\"unneth theorem and the flat factor
in the global symmetric space $D \simeq \HH_{2}\times \HH_{3}\times
\RR$.\footnote{We are being a bit sloppy here, since of course the
quotient of $D$ by $\Gamma$ need not be the product of a quotient of
the $\SL_{2}$-symmetric space $\HH_{2}\times \HH_{3}$ with a quotient
of $\RR$ by a central subgroup.  Indeed, the locally symmetric space
for $\GL_{2}$ will in general be a nontrivial $S^{1}$-bundle over the
locally symmetric space for $\SL_{2}$.  On the other hand, one can
choose subgroups of $\Gamma$ for which this bundle is trivial, and the
K\"unneth theorem then does imply the result. In general, a spectral
sequence for the $S^1$-bundle should give the same result.}

Since the exotic torsion appears to grow and dominates the order of
the $H_{\tors} $, the contributions from $H_{2}$ and $H_{3}$ eventually
\emph{cancel} on average.  This forces the alternating sum to tend to
zero instead of the B-V limit, as seen in Figure \ref{fig:chi23}.  
\begin{figure}[htb]
\begin{center}
\includegraphics[scale=0.75]{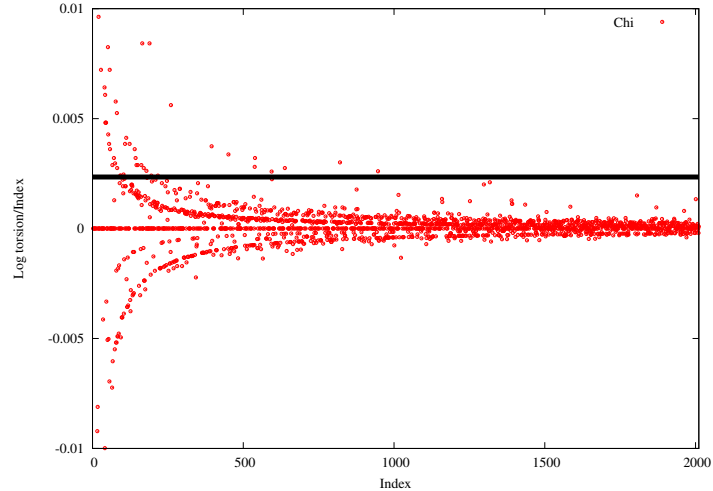}
\end{center}
\caption{The ``Euler characteristic'' for subgroups of $\GL_{2} (\OO_{F})$ for the cubic
field of discriminant $-23$\label{fig:chi23}}
\end{figure}

\section{Towers of subgroups and prime levels}\label{s:towers}

\subsection{}
Conjecture \ref{conj:bv} refers to taking a limit as $\Gamma_{k}$ goes
up a tower of congruence subgroups, whereas our computations
were performed for almost complete lists of congruence
subgroups up to some bound on the level.  Unfortunately, since the
complexity of our computations grows quite rapidly with the level, and
since levels grow quickly in towers, we were unable to systematically
test the dependence of generalizations of Conjecture \ref{conj:bv}
using towers. 

\subsection{}
However, there is one case in which we can give some indications: the
imaginary quadratic field $\QQ (\sqrt{-1})$.  As mentioned in
\S\ref{ss:imag}, our computations overlap with earlier work of
\sengun, and so we do not emphasize them here.  But for
$\GL_{2}(\OO_{\QQ (\sqrt{-1})})$ we were able to compute with very
large levels, which gives us the chance to explore some towers.  The
result is shown in Figure~\ref{fig:tower}, where we give the plot for
six towers each of length six for $H_1$, superimposed on the plot for
$H_1$ for other levels.  For comparison, Figure
\ref{fig:neg1withprimes} shows the same set of data points, with prime
levels indicated.  Figure \ref{fig:neg23withprimes} shows a similar
prime level plot for $H_{2}$ for the $-23$ cubic, and
Figures~\ref{fig:gl3withprimeslevel}--\ref{fig:gl3withprimesindex}
show prime level plots for $H_{2}$ of $\GL_{3} (\ZZ)$.

As one can see, in towers the ratios do appear to climb to the
predicted constant, although they take their time in doing so.  On the
other hand, prime levels seem to do a much better job, in the sense
that the ratios tend to remain closer to the predicted limit than
composite levels do.

\subsection{} \label{subsec:conjectures}
These plots together with Figures
\ref{fig:allgroupsgl3z}--\ref{fig:allgroupsbianchi} suggest several
stronger conjectures than the direct analogue of Conjecture
\ref{conj:bv}: 
\begin{conj}
Let $\Gamma$ be any arithmetic group.  The limit 
\begin{equation}\label{eq:coholimit}
\lim_{k\rightarrow \infty }\frac{\log |H^{i} (\Gamma_{k} ;
\sL)_{\tors}|}{[\Gamma : \Gamma_{k}]}
\end{equation}
should tend to the B-V limit when $\delta = 1$ and when $i$ is at the
top of the cuspidal range and $\Gamma_{k}$ ranges over congruence
subgroups of $\Gamma$ of 
\emph{increasing prime level}.
\end{conj}
\begin{conj}
Let $\Gamma$ be any arithmetic group.
The limit \eqref{eq:coholimit}
should equal the B-V limit as long as $\Gamma_{k}$ ranges over any set
of congruence subgroups of \emph{increasing level}.  In particular, the
lim inf 
\[
\liminf_{\Gamma_{k}}\frac{\log |H^{i} (\Gamma_{k} ;
\sL)_{\tors}|}{[\Gamma : \Gamma_{k}]},
\]
taken over all congruence subgroups, should equal the B-V limit.
\end{conj}

\begin{figure}[htb]
\begin{center}
\includegraphics[scale=0.75]{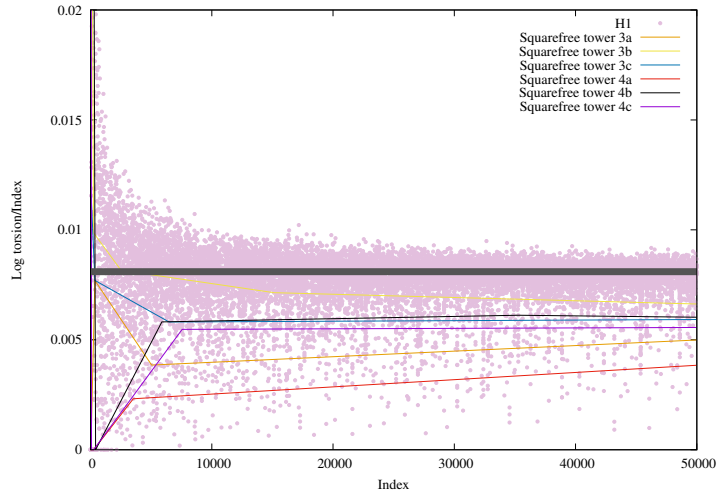}
\end{center}
\caption{Six squarefree level towers of Voronoi homology group $H_1$
  for congruence subgroups for  $\GL_{2} (\ZZ [\sqrt{-1}])$.  Each of
  tower is length six, and the groups in a tower are joined by
  straight lines, superimposed on the plot for $H_1$
  for other levels. 
\label{fig:tower}}
\end{figure}

\begin{figure}[htb]
\begin{center}
\includegraphics[scale=0.75]{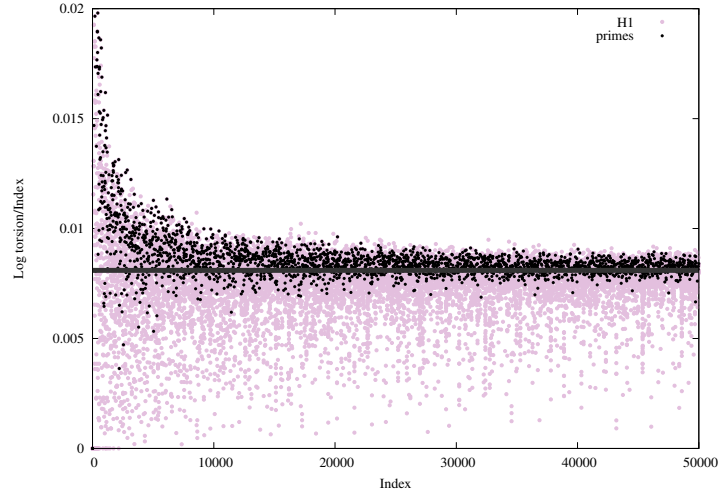}
\end{center}
\caption{$H_{1}$ with prime levels indicated for the subgroups of $\GL_{2} (\ZZ [\sqrt{-1}])$.\label{fig:neg1withprimes}}
\end{figure}

\begin{figure}[htb]
\begin{center}
\includegraphics[scale=0.75]{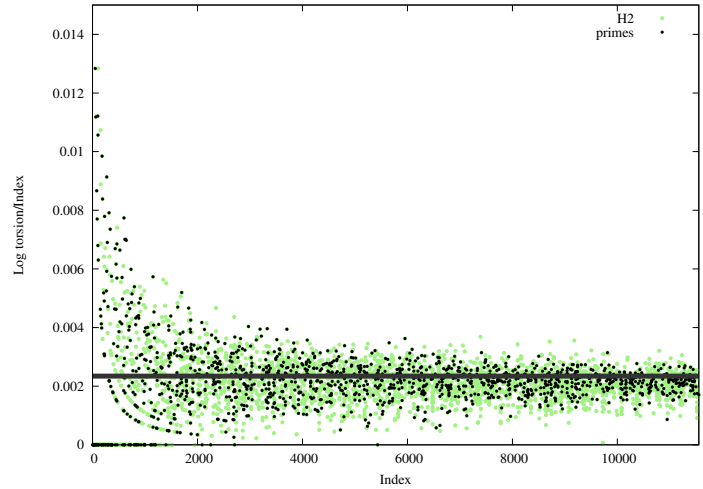}
\end{center}
\caption{$H_{2}$ with prime levels indicated for the subgroups of $\GL_{2} (\OO_{F})$ for the cubic
field of discriminant $-23$.\label{fig:neg23withprimes}}
\end{figure}

\begin{figure}[htb]
\begin{center}
\includegraphics[scale=0.75]{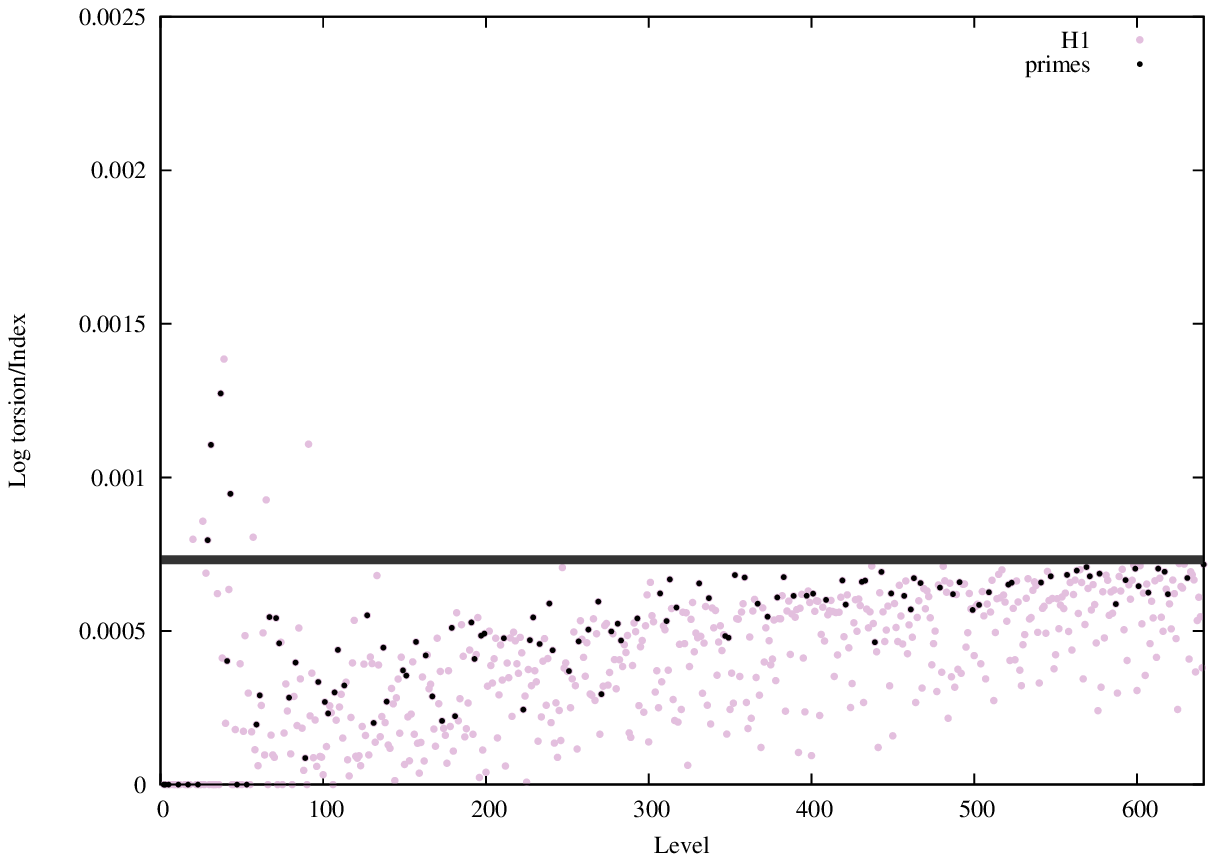}
\end{center}
\caption{$H_{2}$ with prime levels indicated for the subgroups of
$\GL_{3} (\ZZ)$ (ordered by level).\label{fig:gl3withprimeslevel}}
\end{figure}

\begin{figure}[htb]
\begin{center}
\includegraphics[scale=0.75]{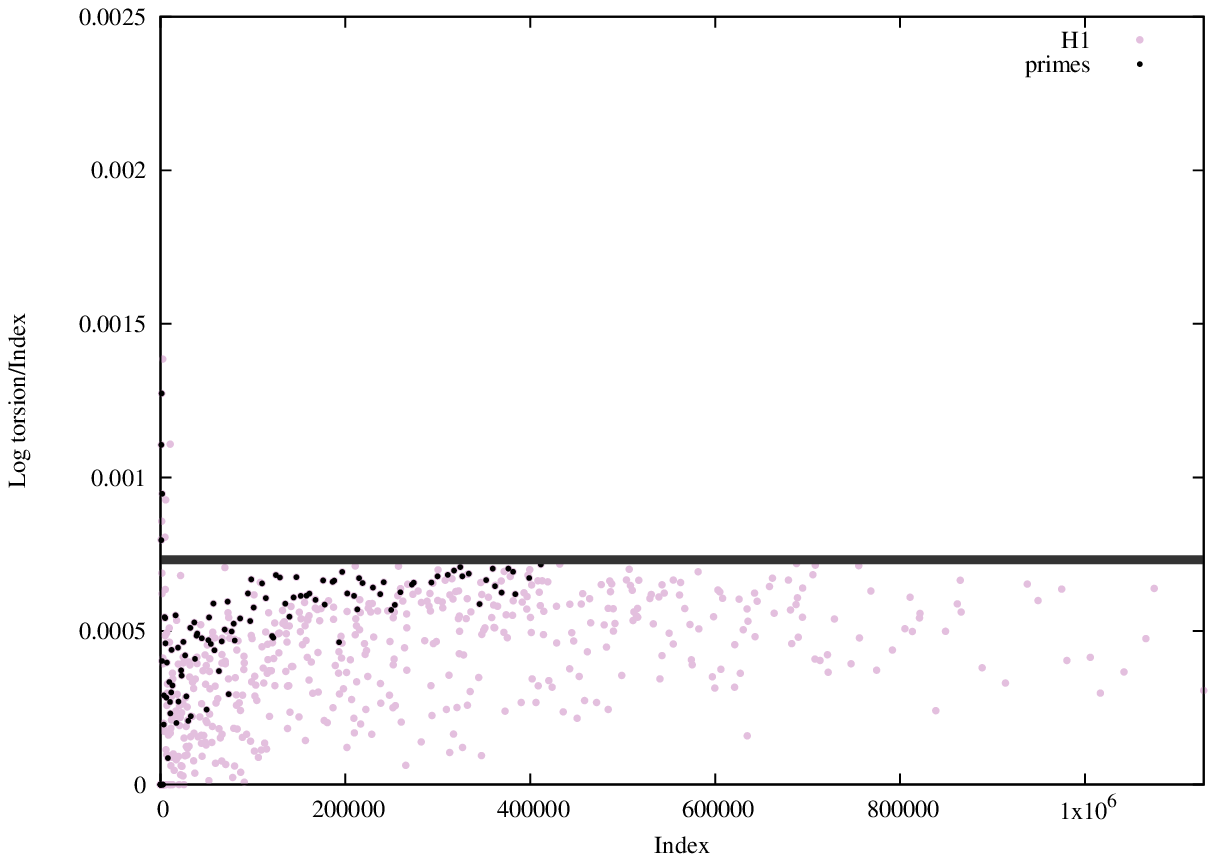}
\end{center}
\caption{$H_{2}$ with prime levels indicated for the subgroups of
$\GL_{3} (\ZZ)$ (ordered by index).\label{fig:gl3withprimesindex}}
\end{figure}

\section{Eisenstein phenomena} \label{s:eisenstein}
\subsection{}
We now report on the Eisenstein cohomology phenomena we investigated.
We begin with an overview of Eisenstein cohomology, the concept of
which is due to Harder; for more information we refer to
\cite{harder.kyoto}.  We restrict ourselves to trivial coefficients.

Recall that $D$ is our global symmetric space.  Let $D^{\BorelSerre}$
be the partial compactification constructed by Borel and Serre
\cite{BS}.  The quotient $Y := \Gamma \backslash D$ is an orbifold,
and the quotient $Y^{\BorelSerre } := \Gamma \backslash
D^{\BorelSerre}$ is a compact orbifold with corners.  We have
\[
H^{*} (\Gamma ; \CC) \simeq H^{*} (Y; \CC) \simeq H^{*}
(Y^{\BorelSerre}; \CC).
\]
Let $\partial Y^{\BorelSerre} = Y^{\BorelSerre}\smallsetminus Y$.  The
inclusion of the boundary $\iota \colon \partial Y^{\BorelSerre}
\rightarrow Y^{\BorelSerre}$ induces a map on cohomology
$\iota^{*}\colon H^{*} (Y^{\BorelSerre}; \CC) \rightarrow H^{*}
(\partial Y^{\BorelSerre}; \CC)$.  Moreover, this map is compatible
with the action of the Hecke operators: the Hecke operators act
naturally on the boundary $\partial Y^{\BorelSerre}$.  The kernel
$H^{*}_{!}  (Y^{\BorelSerre}; \CC)$ of $\iota^{*}$ is called the
\emph{interior cohomology}; it equals the image of the cohomology with
compact supports.  The goal of Eisenstein cohomology is to use
Eisenstein series and cohomology classes on the boundary to construct
a Hecke-equivariant section $s\colon H^{*} (\partial Y^{\BorelSerre};
\CC) \rightarrow H^{*} (Y^{\BorelSerre}; \CC)$ mapping onto a
complement $H^{*}_{\Eis} (Y^{\BorelSerre}; \CC )$ of the interior
cohomology in the full cohomology.  We call classes in the image of
$s$ \emph{Eisenstein classes}. 

\subsection{}
The construction of Eisenstein cohomology provides a link between the
cohomology of the Borel--Serre boundary $\partial Y^{\BorelSerre}$ and
the cohomology of $Y$.  In general it is very difficult to fully
analyze the part of the cohomology coming from the boundary, and for
groups of $\QQ$-rank $>1$ (such as $\GL_{n}/\QQ$, $n\geq 3$), complete
results are not known.  However, one can typically predict the types
of relationships one might see, and can observe them in practice by
computing the Hecke action.  For example, in the case of constant
coefficients, one experimentally sees cohomology classes in $H^{5}$ of
congruence subgroups of $\GL_{4}/\QQ$ that correspond to classes in
$H^{3}$ of congruence subgroups of $\GL_{3}/\QQ$, as well as classes
corresponding to weight 2 and weight 4 holomorphic modular forms
\cite{AGM}; one can refer there to see descriptions of the
predicted mechanisms for the Eisenstein lifting.

In our experiments, we have one pair of groups where we can hope to
see Eisenstein phenomena: $\GL_3/\QQ$ and $\GL_4/\QQ$.  Since we are
not computing Hecke operators on the torsion classes we computed, we
instead try to see connections by looking for exotic primes that
appear in different cohomological degrees.  In particular, if we see
the same exotic prime occurring as torsion in a pair of cohomology
groups for $(\GL_{3}, \GL_{4})$, and these cohomology groups can be
related by a standard Eisenstein mechanism, we take that as evidence
of Eisenstein phenomena.  As before, the
torsion size is given in factored form with exotic torsion in
\textbf{bold}. 

\subsection{}
The first collection of Eisenstein lifts goes from $H_{2}$ of
$\GL_{3}$ to $H_{3}$ of $\GL_{4}$.  Recall that we index groups
homologically by the Voronoi complex \eqref{eq:voronoiiso}.  Thus
$H_{2}$ of $\GL_{3}$ refers to $H^{3}$ and $H_{3}$ of $\GL_{4}$ refers
to $H^{6}$; both of these are in the vcd of their respective groups.
\begin{itemize}
\item At level $114$, the size of the torsion in $H_3$ is
  $2^{12} \cdot 3^{7} \cdot \mathbf{11^{4}}$.
  The corresponding torsion for $\GL_3$ in $H_2$ is
  $2^{5} \cdot 3^{3} \cdot \mathbf{11^{2}}$.   
\item At level $118$, the size of the torsion in $H_3$ is
  $2^{14} \cdot \mathbf{17^{4}}$.  The
  corresponding torsion for $\GL_3$ in $H_2$ is
  $\mathbf{17^{2}}$. 
\item At level $119$, the size of the torsion in $H_3$ is $2^{4} \cdot
  3^{3} \cdot \mathbf{31^{4}}$.  The corresponding
  torsion for $\GL_3$ in $H_2$ is $2^{2} \cdot 3^{1}
  \cdot \mathbf{31^{2}}$.
\end{itemize}

\subsection{}
We also see lifts for $H_3$ of $\GL_{3}$ to $H_4$ for $\GL_4$; both of
these correspond to cohomological degree one below the vcd of their
respective groups. 
\begin{itemize}
\item At level $49$, the size of the torsion in $H_4$ is $3^{1} \cdot
  \mathbf{7^{2}}$.  The corresponding torsion for $\GL_3$ in $H_3$ is
  $\mathbf{7}$. 
\item At level $98$, the size of the torsion in $H_4$ is $\mathbf{7^5}$.
  The corresponding torsion for $\GL_3$ in $H_3$ is $\mathbf{7}$.
\end{itemize}

Since Eisenstein series are differential forms defined over the
complex numbers, they cannot directly be used to study torsion.  
Instead, a topological analysis of the Borel-Serre boundary is
required, which can mimic the Eisenstein phenomena, by showing how
certain cohomology classes on the Borel-Serre boundary can be lifted
to cohomology classes on the whole space.  The topology could be
carried out with integral 
coefficients and thus could deal with the torsion in theory.
This has not been done for the boundary
in the case of $\GL_4$.  Such a topological analysis poses a very
interesting and probably very hard problem.  Heuristically we can say
that both of these examples can be explained by the topological Eisenstein
mechanism of placing a class on $H^{k} (\partial_{P}Y )$, where
$\partial _{P }Y$ is the Levi part of Borel--Serre boundary component
corresponding to a maximal parabolic subgroup of type $(1,3)$.  Such a
class has the potential of lifting to $H^{k+3} (Y^{\BorelSerre})$
through a spectral sequence with $E_{2}$ page given by $H^{i}
(\partial_{P}Y, H^{j} N)$, where $N$ is the nilmanifold part of the
Borel--Serre boundary component.  We do not know why we observe
Eisenstein phenomena at these levels and not at other levels. 

\section{Conclusions and further questions}\label{s:conclusions}

In this paper, we computed the torsion in the Voronoi homology of
congruence subgroups of several arithmetic groups.  The Voronoi
homology is isomorphic to the group cohomology in dual dimensions.
Our examples treated groups of deficiencies $1$ and $2$.

\begin{itemize}
\item We found excellent agreement in our results with the general
heuristic espoused by Bergeron--Venkatesh \cite{bv}, namely that groups
with deficiency $1$ should have exponential growth in the torsion in
their cohomology.  We also found excellent quantitative agreement with
the predicted asymptotic limit from Conjecture \ref{conj:bv}, suitably
interpreted for reductive groups.
\item We found that, when the $\QQ$-rank of a group is $>0$ and the
deficiency is $1$, the explosive exotic torsion occurs in the top
cohomological degree of the cuspidal range.  
\item When the deficiency is $>1$, we still found exotic torsion in
the top degree of the cuspidal range.  However, the growth rate of the
size of the torsion subgroup appears much lower than that in the
deficiency $1$ case.  It would be interesting to formulate a
quantitative estimate for the growth of the torsion in this case.
Would this estimate be polynomial or subexponential?
\item For groups of deficiency $1$, the growth of the torsion in
towers of congruence subgroups seems to agree with the predicted
asymptotic limit, although oddly the convergence seems significantly slower
than that experienced by families of congruence subgroups of increasing
prime level or simply the family of all congruence subgroups ordered
by increasing level.  
\item The exotic torsion in a group of deficiency $1$ appears to tend
to transfer to another via Eisenstein cohomology.  What is the
explanation of when this transfer happens and when it doesn't?  Could
this be related to divisibility of special values of some L-function
by the exotic primes in question?
\item We made two new conjectures along the lines of
  Conjecture~\ref{conj:bv} but for not necessarily cocompact groups
  and for different families of congruence subgroups.  These may be
  found in \S\ref{subsec:conjectures}. 
\end{itemize}

\bibliography{tonsoftorsion_xxx}    
\end{document}